\documentclass[12pt,draftcls,onecolumn]{IEEEtran}
\usepackage[ruled,boxed]{algorithm2e}

\usepackage{amsmath,graphicx,color,psfrag,mathrsfs,bm}
\usepackage{hyperref}
\usepackage{longtable}          
\usepackage{amsthm}
\usepackage{amssymb}
\usepackage{algorithmic}

\usepackage{multicol}
\newtheorem{thm}{Theorem}[section]
\newtheorem{kor}[thm]{Corollary}
\newtheorem{lem}[thm]{Lemma}
\newtheorem{prop}[thm]{Proposition}
\newtheorem{assumption}[thm]{Assumption}

\theoremstyle{definition}
\newtheorem{defn}{Definition}[section]
\theoremstyle{remark}
\newtheorem{rem}{Remark}[section]

\begin{document}
\title{From Asymptotic to Finite-Sample Minimax Robust Hypothesis Testing}

\author{%
  \IEEEauthorblockN{G\"okhan  G\"ul}\\
  \IEEEauthorblockA{Preventive Cardiology and Preventive Medicine, Department of Cardiology, University Medical Center of the Johannes Gutenberg University Mainz\\
                    Clinical Epidemiology and Systems Medicine, Center for Thrombosis and Hemostasis, University Medical Center  Johannes Gutenberg University Mainz\\
                    German Center for Cardiovascular Research (DZHK), Partner Site Rhine Main, University Medical Center of the Johannes Gutenberg University Mainz\\
                    Langenbeckstra\ss e 1, 55131 Mainz, Germany\\
                    Email: goekhan.guel@unimedizin-mainz.de}
}

\maketitle

\begin{abstract}
This paper establishes a formal connection between finite-sample and asymptotically minimax robust hypothesis testing under distributional uncertainty. It is shown that, whenever a finite-sample minimax robust test exists, it coincides with the solution of the corresponding asymptotic minimax problem. This result enables the analytical derivation of finite-sample minimax robust tests using asymptotic theory, bypassing the need for heuristic constructions. The total variation distance and band model are examined as representative uncertainty classes. For each, the least favorable distributions and corresponding robust likelihood ratio functions are derived in parametric form. In the total variation case, the new derivation generalizes earlier results by allowing unequal robustness parameters. The theory also explains and systematizes previously heuristic designs. Simulations are provided to illustrate the theoretical results.
\end{abstract}


\IEEEpeerreviewmaketitle


\section{Introduction}
In simple binary hypothesis testing, the design of optimum tests requires complete statistical knowledge of the underlying data distributions \cite{kay}. However, this assumption is often too restrictive and seldom holds in practice \cite{levy}. A pragmatic alternative is to adopt either parametric \cite{kassam,elsawy} or non-parametric robust approaches \cite{nonparametric}. Parametric models, including those based on $M$-estimators \cite{hube81_2}, implicitly assume that the general form of the distributions remains known, whereas non-parametric methods, such as the sign and Wilcoxon tests, make only mild assumptions about the underlying distributions and are therefore regarded as more conservative approaches \cite{Wilcoxon,gulbook}.\\
Minimax robust hypothesis testing (MRHT) offers an intermediate framework between parametric and non-parametric methods. In MRHT, the true distribution of the observed data, $G_j$, is assumed to belong to an uncertainty class $\mathscr{G}_j$, whose size is determined by a robustness parameter $\epsilon_j$. By adjusting these parameters, one can balance robustness against detection power. The choice of uncertainty classes is typically application-dependent, with common formulations being model-based (e.g., $\epsilon$-contamination models) or distance-based (e.g., sets defined through the KL-divergence) \cite{levy09}. The designer’s objective is then to determine a decision rule $\hat{\delta}$ that minimizes a predefined risk function evaluated under the least favorable distributions (LFDs)\footnote{For brevity, no notational distinction is made between least favorable distributions and their corresponding densities unless confusion may arise.}. Under mild regularity conditions, such tests are optimal in the minimax sense, guaranteeing the best possible detection performance under the assumed model uncertainties.\\
Robustness can also be characterized in terms of sample size. Depending on the underlying uncertainty model, minimax robust tests may exist either for finite samples or only in the asymptotic regime. For example, in the $\epsilon$-contamination model and total variation neighborhoods, minimax robust tests exist for any finite sample size~\cite{hube65,HuberStrassen1973}, whereas for formulations based on divergence measures, only asymptotically minimax robust solutions are guaranteed~\cite{dabak,moment}. This distinction motivates the development of a unified framework that systematically links finite- and large-sample robustness.

\subsection{Related work}
The foundations of robust hypothesis testing trace back to the seminal works of P.~J.~Huber, who in 1965 introduced a robust version of the likelihood ratio test for the $\epsilon$-contamination and total variation classes of probability distributions~\cite{hube65}. Huber derived the corresponding least favorable distributions and showed that the clipped likelihood ratio test constitutes the minimax robust solution for both uncertainty classes. This line of research was later extended by Huber and Strassen~\cite{hube68} and further generalized to the framework of $2$-alternating capacities~\cite{hube73}, establishing the theoretical foundation of minimax robust detection. Nevertheless, it was later shown that minimax robust tests do not always exist, for instance when uncertainty classes are defined using the KL-divergence~\cite{gul5}.\\
One natural extension of the basic uncertainty formulations is to construct classes that incorporate partial prior knowledge about the true distributions, such as their approximate shapes, supports, or low-order moments. Among these, moment classes specify the uncertainty in terms of finitely many statistical moments of the distributions~\cite{moment}, allowing direct control of mean and variance deviations while maintaining analytical tractability. The $p$-point classes~\cite{elsawy,vastola} provide an alternative description by partitioning the probability domain into $p$ disjoint regions with prescribed probability masses, which enables flexible modeling of multimodal or asymmetric deviations. Band models, originally introduced by Kassam~\cite{kassamband} and later refined by \mbox{Fau\ss}~et al.~\cite{fauss}, instead define upper and lower bounds on the admissible densities, thereby encoding approximate shape and location information. These formulations offer complementary ways of representing partial distributional knowledge and serve as a foundation for many subsequent developments in robust detection and estimation.\\
Several studies have applied these uncertainty formulations to practical detection problems and extended them toward data-driven settings. Early works such as~\cite{martin,kassam2,martin2} investigated clipped-likelihood and mixture-based detectors under Gaussian and contaminated noise models. Variants of moment and $p$-point uncertainty sets have been used across disciplines, including robust decision-making in finance~\cite{smith}, admission control~\cite{brichet}, and queueing systems~\cite{johnson}. More recent lines of research have connected these formulations with distributionally robust optimization~\cite{RahimianMehrotra2019}, enabling tractable convex approximations under Wasserstein~\cite{GaoXie2018} and Sinkhorn-type metrics~\cite{WangGaoXie2024}. Related extensions include kernel-based uncertainty sets~\cite{SunZou2022,SchrabKim2024}, empirical distribution approaches~\cite{WangGaoXie2022}, and adversarially robust testing frameworks~\cite{PuranikMadhowPedarsani2021}. Together, these developments highlight a continuous progression towards more comprehensive and versatile frameworks for robust decision-making under distributional uncertainty.

\subsection{Summary of the paper and its contributions}
In our earlier work, we showed that asymptotically minimax robust tests can be systematically designed by solving the optimization problem  
\begin{equation*}
\min_{u\in(0,1)} \max_{(G_0,G_1)\in\mathscr{G}_0\times \mathscr{G}_1}D_u(G_0,G_1),
\end{equation*}  
where  
\begin{equation*}
D_u(G_0,G_1)=\int_{\Omega} g_1^u g_0^{1-u} \, d\mu
\end{equation*}  
is the so-called $u$-affinity \cite{gularxiv}.\\  
In this paper, it is demonstrated that asymptotic minimax theory not only provides a principled framework in its own right, but also yields a direct analytical route to derive finite-sample minimax robust tests — whenever such tests exist. This connection enables the exact computation of LFDs without requiring heuristic constructions. Moreover, the same theoretical foundation also allows for numerical computation of LFDs in cases where analytical forms are not available; moment and p-point uncertainty classes are treated as two representative examples. The main contributions of this work are as follows:
\begin{enumerate}
\item It is proven that whenever a finite-sample minimax robust test exists, it coincides with the test derived from asymptotic minimax theory (see Proposition~\ref{thm:AMR_implies_FMR_conditional}). This unifies the treatment of finite and infinite sample regimes.
\item Based on this equivalence, finite-sample minimax robust tests are analytically derived for two uncertainty classes: total variation distance and band model. In each case, the least favorable distributions and robust likelihood ratio functions (LRFs) are obtained in parametric form.
\item For the total variation case, the derived LFDs generalize Huber's results \cite{hube65} by allowing asymmetric robustness levels, i.e., different contamination radii $\epsilon_0\neq \epsilon_1$ for the two hypotheses $\mathcal{H}_0$ and $\mathcal{H}_1$ (see Theorem~\ref{theorem31}).
\item It is shown that two special cases of the band model yield distinct versions of the $\epsilon$-contamination neighborhood, each of which admits a clipped likelihood ratio test (CLRT) as the minimax robust solution (see Theorems~\ref{theorem1} and \ref{theorem2}). One of these cases was previously studied by Huber \cite{hube65}; the other is newly derived in this work. It is also proven that both are single-sample minimax robust (see Theorem~\ref{theoremeps}), and their intersection corresponds to the general form of the band model.
\item In the symmetric total variation case, our formulation reveals that the clipping thresholds necessarily satisfy $t_l t_u=1$, implying that the family of least favorable distributions is intrinsically one-dimensional. This observation clarifies and sharpens the classical construction of Huber, in which the symmetry-induced parameter reduction is not made explicit.
\item Taken together, these results provide a theoretical foundation for several previously heuristic constructions, such as the designs proposed by Huber \cite{hube65} and Kassam \cite{kassamband}, by showing how they emerge naturally from the asymptotic theory.
\end{enumerate}

\subsection{Outline of the paper}
The rest of the paper is organized as follows. In Section~\ref{sec2}, single-sample, finite-sample and asymptotic minimax robustness are defined, the connection between asymptotic and finite-sample minimax robustness is established and the problem statement is made. In Section~\ref{sec3}, least favorable distributions and the correesponding robust likelihood functions are derived in closed parametric forms for the total variation neighborhood and the band model. The parameters can be obtained by solving a system of non-linear equation, leading to the asymptotic minimax robust tests. In Section~\ref{sec4} the formulation of asymptotic minimax robustness is extended to the LFDs which cannot be obtained in closed form and numerical methods are required. Moment classes and p-point classes are presented as the two examples. In Section~\ref{sec5}, simulations are performed to evaluate and exemplify the theoretical derivations. Finally in Section~\ref{sec6}, the paper is concluded.

\subsection{Notations}
The following notations are applied throughout the paper. Upper case symbols are used for probability distributions and random variables, and the corresponding lower case symbols denote the density functions and observations, respectively. Boldface symbols are used for the sequence of random variables, sequence of observations or joint functions. The hypotheses $\mathcal{H}_0$ and $\mathcal{H}_1$ are associated with the nominal probability measures $F_0$ and $F_1$, whereas the corresponding actual distributions are denoted by $G_0$ and $G_1$. The sets of probability distributions are denoted by $\mathscr G_0$ and $\mathscr G_1$. Every probability measure, e.g.\ $G[\cdot]$, is associated with its distribution function $G(\cdot)$ and the density function $g(\cdot)$. The symbol $\mathscr{M}$ denotes the set of all distribution functions on $\Omega$. The notation $\hat{(\cdot)}$ indicates the least favorable distributions $\hat {G}_j\in \mathscr G_j$, the corresponding densities $\hat{g}_j$, or the robust likelihood ratio test $\hat{\delta}$. The expected value of a random variable $Y\sim G_j$ is denoted by $\mathbb{E}_{G_j}[Y]$. The argument (value on the domain) of the subsequent operation is denoted by $\arg$.

\section{Linking Asymptotic Minimax Theory to Finite-Sample Tests}\label{sec2}
Let $\mathscr{M}$ denote the set of all probability measures on $\Omega$, and let $\mathscr{G}_0 \subset \mathscr{M}$ and $\mathscr{G}_1 \subset \mathscr{M}$ be two uncertainty classes associated with the hypotheses $\mathcal{H}_0$ and $\mathcal{H}_1$, respectively. Consider a sequence of $n$ independent and identically distributed (i.i.d.) random variables $\boldsymbol{Y} = (Y_1,\ldots,Y_n)$, each taking values in $\Omega$. Under the hypothesis $\mathcal{H}_j$, the distribution of $Y_k$ is denoted by $G_j \in \mathscr{G}_j$, and the binary decision problem is to determine which of the hypotheses
\begin{align}
\mathcal{H}_0 &: Y_k \sim G_0,\quad G_0 \in \mathscr{G}_0, \nonumber\\
\mathcal{H}_1 &: Y_k \sim G_1,\quad G_1 \in \mathscr{G}_1,
\end{align}
is true for $k \in \{1,\ldots,n\}$.
Let $\delta : \boldsymbol{Y} \mapsto \{0,1\}$ denote a decision rule, with false-alarm probability $P_F = G_0[\delta(\boldsymbol{Y})=1]$ and miss-detection probability $P_M = G_1[\delta(\boldsymbol{Y})=0]$. For a given prior probability $\pi_0 = P(\mathcal{H}_0)$, the error probability is defined as
\begin{equation}
P_E(\delta,G_0,G_1)
    = \pi_0 P_F(\delta,G_0) + (1 - \pi_0) P_M(\delta,G_1).
\end{equation}
The minimax formulation seeks a decision rule $\hat{\delta}$ together with least favorable distributions $(\hat{G}_0,\hat{G}_1)$ that solve
\begin{equation}\label{eq10}
\min_{\delta} \max_{(G_0,G_1)\in \mathscr{G}_0\times\mathscr{G}_1}
P_E(\delta,G_0,G_1).
\end{equation}
Let $\hat{l} = \hat{g}_1 / \hat{g}_0$ denote the likelihood ratio induced by the least favorable distributions, define
$X_k = \log \hat{l}(Y_k)$, and let the empirical mean be
\[
S_n(\boldsymbol{X}) = \frac{1}{n}\sum_{k=1}^n X_k.
\]
Since the minimax problem \eqref{eq10} is solved by a likelihood ratio test based on $(\hat{G}_0,\hat{G}_1)$, the associated decision rule takes the form
\begin{equation}\label{eq115}
\hat{\delta}(\boldsymbol{X}) = 
\begin{cases}
0, & S_n(\boldsymbol{X}) < t,\\[2pt]
1, & S_n(\boldsymbol{X}) \geq t,
\end{cases}
\end{equation}
where $t \in \mathbb{R}$ is a threshold chosen to minimize the worst-case error probability.\\  
In the remainder of this section the asymptotic setting $n \to \infty$ is considered. The behaviour of $S_n(\boldsymbol{X})$ under $G_0$ and $G_1$, together with its associated large-deviation properties, provides the foundation for connecting asymptotic minimax robustness with finite-sample minimax robustness in the subsequent subsections. 
\subsection{Minimax Robustness Concepts}

\subsubsection{Single-sample minimax robustness (SMR)}
Single-sample minimax robustness characterizes optimality against worst-case
distributional perturbations when only a \emph{single} observation is available.

\begin{defn}[Single-sample minimax robustness]\label{defn:SMR}
Let $n=1$ and let $Y=Y_1$. Suppose there exist distributions 
$\hat{G}_0\in \mathscr{G}_0$ and $\hat{G}_1\in \mathscr{G}_1$ such that the 
likelihood ratio $\hat{l}(Y)=\hat{g}_1(Y)/\hat{g}_0(Y)$ satisfies
\begin{align}\label{eq:SMR-ineq}
G_0\!\left[\hat{l}(Y)<t\right] &\;\ge\;
\hat{G}_0\!\left[\hat{l}(Y)<t\right], \nonumber\\
G_1\!\left[\hat{l}(Y)<t\right] &\;\le\;
\hat{G}_1\!\left[\hat{l}(Y)<t\right],
\end{align}
for all $t\in\mathbb{R}$ and all $(G_0,G_1)\in \mathscr{G}_0\times \mathscr{G}_1$. Then the likelihood–ratio test induced by $(\hat{G}_0,\hat{G}_1)$, which solves the single-sample minimax problem~\eqref{eq10}, is called \emph{single-sample minimax robust (SMR)} \cite[p.~1754]{hube65}.
\end{defn}
The existence of LFDs depends on the choice of uncertainty classes $\mathscr{G}_0$ and $\mathscr{G}_1$. For instance, SMR tests exist for $\epsilon$-contamination and total variation classes~\cite{hube65}, whereas they fail to exist for KL-divergence-based neighborhoods~\cite{gul5}. Even if no minimax solution exists within the class of deterministic decision rules, a unique minimax rule can still be obtained over randomized decision rules~\cite{gul5}. However, the information contained in the randomization does not carry over to products of likelihood ratios, and therefore cannot be directly extended to multiple samples. 

\subsubsection{Finite-sample minimax robustness (FMR)}
A test is called \emph{finite-sample minimax robust (FMR)} if it is minimax robust for every finite sample size $n<\infty$; equivalently, the single-sample stochastic ordering conditions~\eqref{eq:SMR-ineq} hold for the joint likelihood ratio based on $\boldsymbol{Y}$.

\subsubsection{Asymptotic minimax robustness (AMR)}
In the asymptotic regime, the behavior of $S_n$ is characterized through large-deviation principles. The minimax criterion compares the exponential decay rates of the error probabilities under all $(G_0,G_1)\in\mathscr{G}_0\times\mathscr{G}_1$ with those achieved under a candidate pair of least favorable distributions. The following regularity assumptions ensure that the comparison of error exponents is meaningful for all distributions in the uncertainty classes.

\begin{assumption}[Uniform exponential moments]\label{asm1}
There exists $\varepsilon>0$ such that the random variable $X_k = \log(\hat{l}(Y_k))$ has a finite moment generating function (MGF)
\begin{equation}
\mathbb{E}_{G_j}\!\left[e^{u X_k}\right] < \infty,
\qquad \text{for all } |u|<\varepsilon,\; j\in\{0,1\},
\end{equation}
for every $(G_0,G_1)\in\mathscr{G}_0\times\mathscr{G}_1$. This ensures that the corresponding Cramér rate functions are finite and well-defined.
\end{assumption}

\begin{assumption}[Threshold separation]\label{asm2}
There exists a real number $t$ such that
\begin{equation}\label{eq:thresh-sep}
\mathbb{E}_{G_0}[X_k] < t < \mathbb{E}_{G_1}[X_k],
\end{equation}
for all $(G_0,G_1)\in\mathscr{G}_0\times\mathscr{G}_1$.
This guarantees that both false-alarm and miss-detection exponents are positive.
\end{assumption}

\begin{kor}\label{corollary1}
Any of the following conditions is sufficient for \eqref{eq:thresh-sep} to hold:
\begin{enumerate}
\item There exist LFDs $(G_0,G_1)\in\mathscr{G}_0\times\mathscr{G}_1$ satisfying single-sample minimax robustness. 
\item $\mathbb{E}_{\hat{G}_1}[X_k]<\mathbb{E}_{G_1}[X_k]$ and $\mathbb{E}_{\hat{G}_0}[X_k]>\mathbb{E}_{G_0}[X_k]$ for all $(G_0,G_1)\in \mathscr{G}_0\times \mathscr{G}_1$
\item  $\min\mathbb{E}_{G_1}[X_k]>\max\mathbb{E}_{G_0}[X_k]$ for all $(G_0,G_1)\in \mathscr{G}_0\times \mathscr{G}_1$
\end{enumerate}
Moreover, $1\Longrightarrow 2 \Longrightarrow 3$, and neither implication is reversible in general.
\end{kor}

\begin{defn}[Asymptotic minimax robustness]\label{def:AMR}
A test is called \emph{asymptotically minimax robust} if, for a threshold $t$ satisfying Assumption~\ref{asm2} and under Assumption~\ref{asm1}, the inequalities
\begin{equation}
\lim_{n\rightarrow\infty}
\frac{1}{n}\log G_0\!\left[S_n(\boldsymbol{X})>t\right]
\;\le\;
\lim_{n\rightarrow\infty}
\frac{1}{n}\log \hat{G}_0\!\left[S_n(\boldsymbol{X})>t\right],
\end{equation}
and
\begin{equation}
\lim_{n\rightarrow\infty}
\frac{1}{n}\log G_1\!\left[S_n(\boldsymbol{X})\le t\right]
\;\le\;
\lim_{n\rightarrow\infty}
\frac{1}{n}\log \hat{G}_1\!\left[S_n(\boldsymbol{X})\le t\right],
\end{equation}
hold for all $(G_0,G_1)\in\mathscr{G}_0\times\mathscr{G}_1$.
\end{defn}

\subsection{Relations Between Finite-Sample and Asymptotic Minimax Robustness}\label{sec2_2}
This section develops the connection between classical single-sample minimax robustness (SMR), its finite-sample extension (FMR), and the asymptotic minimax robustness (AMR). We show that SMR and FMR coincide, that FMR always implies AMR, and that AMR implies FMR under a mild uniqueness condition.
\begin{thm}\label{prop1}
If a single-sample minimax robust test exists, then a minimax robust test exists for every finite sample size $n<\infty$, with the same least favorable distributions and for the same uncertainty classes. Hence,
\begin{equation}
\text{SMR} \;\Longleftrightarrow\; \text{FMR}.
\end{equation}
\end{thm}

\begin{IEEEproof}
See~\cite[Section~4]{hube65}.
\end{IEEEproof}

The next result recalls a fundamental equivalence between the stochastic ordering condition~\eqref{eq:SMR-ineq} and the minimization of $f$-divergences.

\begin{thm}[Equivalence of SMR and $f$-divergence minimization]\label{thm1}
Let $G_0$ and $G_1$ be distributions absolutely continuous with respect to a common measure~$\mu$ on~$\Omega$. For the $f$-divergence
\begin{equation}
D_f(G_0,G_1)=\int_\Omega f\!\left(\frac{g_0}{g_1}\right) g_1\, d\mu,
\end{equation}
where $f:\mathbb R_{\ge 0}\!\to\!\mathbb R$ is convex and satisfies $f(1)=0$, the following equivalence holds:
\begin{equation}\label{eq13}
(\hat G_0,\hat G_1) \text{ satisfies \eqref{eq:SMR-ineq}} 
\quad\Longleftrightarrow\quad
(\hat G_0,\hat G_1) \text{ minimizes } D_f
\end{equation}
over all $(G_0,G_1)\in\mathscr G_0\times\mathscr G_1$, for every twice continuously differentiable convex $f$.
\end{thm}

\begin{IEEEproof}
See Appendix~\ref{appendix0}.
\end{IEEEproof}

\begin{thm}[FMR implies AMR]\label{thm:FMR_implies_AMR}
Let $(\mathscr G_0,\mathscr G_1)$ be uncertainty classes for which a finite-sample minimax robust (FMR) test exists with LFDs $(\hat G_0,\hat G_1)$. Then an asymptotically minimax robust (AMR) test exists for the same uncertainty classes and with the same LFDs. In particular, 
\begin{equation}
\mathrm{FMR}\;\Longrightarrow\;\mathrm{AMR}.
\end{equation}
\end{thm}

\begin{IEEEproof}
Let $(\hat G_0,\hat G_1)$ be the LFDs generating the FMR test. By Definition~\ref{defn:SMR} and Theorem~\ref{prop1}, these distributions satisfy the stochastic ordering condition~\eqref{eq:SMR-ineq}. From Theorem~\ref{thm1}, the same pair minimizes every twice differentiable convex $f$-divergence $D_f$, and in particular minimizes $-D_u$ (equivalently maximizes $D_u$) for every $u\in(0,1)$, as well as $D_{\mathrm{KL}}$. The AMR definition requires two conditions: (i) the threshold separation condition~\eqref{eq:thresh-sep}, and (ii) LFDs that maximize $D_u$ for the minimizing $u$. Condition~(i) holds by Corollary~\ref{corollary1}, applied to the LFDs $(\hat G_0,\hat G_1)$. Condition~(ii) holds because $(\hat G_0,\hat G_1)$ maximize $D_u$ for all $u\in(0,1)$ by Theorem~\ref{thm1}. Hence the same LFDs satisfy the AMR requirements, completing the proof.
\end{IEEEproof}

\begin{rem}
The AMR formulation requires the finite–MGF condition of Assumption~\ref{asm1}. If the uncertainty classes $\mathscr G_0,\mathscr G_1$ contain pathological elements (e.g., point masses located at zeros of the nominal densities) that produce infinite MGFs, these distributions are never least favorable: they yield infinite $f$–divergences or zero Chernoff-type exponents, and therefore cannot minimize any convex $f$–divergence or maximize $D_u$. Such elements may be removed from the uncertainty classes without affecting the LFDs, so the effective uncertainty classes (i.e., those that contain all potential minimizers of the robust problem) automatically satisfy Assumption~\ref{asm1}. Thus the LFDs $(\hat G_0,\hat G_1)$ obtained from the FMR problem also satisfy the AMR regularity requirements. An analogous implicit restriction is already present in Huber’s SMR inequalities: any pair of distributions for which $g_0$ vanishes on a set where $G_1$ places positive mass leads to nonintegrable likelihood ratios or divergent integrals, and therefore cannot be least favorable.
\end{rem}

\begin{thm}[AMR implies FMR conditionally]\label{thm:AMR_implies_FMR_conditional}
Let $(\mathscr G_0,\mathscr G_1)$ be uncertainty classes for which an AMR solution exists. Let $(\overline G_0,\overline G_1)$ be the AMR LFDs corresponding to the minimizing parameter $u^*\in(0,1)$, i.e.\ they maximize $D_{u^*}$ over $\mathscr G_0\times\mathscr G_1$. Assume that for $u^*$ the maximization of $D_{u^*}$ admits a unique solution (up to $\mu$-a.e.\ equality).  
If a finite-sample minimax robust (FMR) test exists for the same uncertainty classes, then its LFDs must coincide a.e.\ with $(\overline G_0,\overline G_1)$. Thus, whenever FMR exists,
\begin{equation}
\mathrm{AMR}\;\Longrightarrow\;\mathrm{FMR}.
\end{equation} 
\end{thm}

\begin{IEEEproof}
Let $(\hat G_0,\hat G_1)$ denote the LFDs of an FMR test. By Theorem~\ref{thm1}, they maximize $D_{u^*}$, since FMR implies minimization of all $f$-divergences. Because the $D_{u^*}$-maximizer is unique by assumption, any two maximizers must coincide $\mu$-a.e. Thus $(\hat G_0,\hat G_1) = (\overline G_0,\overline G_1)$ almost everywhere.
\end{IEEEproof}

\begin{rem}
If the maximizer of $D_{u^*}$ over $\mathscr G_0\times\mathscr G_1$ is not unique, the above argument shows only that any FMR LFD pair must belong to the set of $D_{u^*}$-maximizers.  
In particular, FMR LFDs coincide (up to $\mu$-a.e.\ equality) with \emph{some} AMR LFDs, but not necessarily with a distinguished pair $(\overline G_0,\overline G_1)$.
\end{rem}


The uniqueness assumption in Theorem~\ref{thm:AMR_implies_FMR_conditional} is in fact mild; under natural convexity and compactness conditions, $D_u$ admits a unique maximizer, as formalized in the following lemma.

\begin{lem}[Uniqueness of $D_u$-maximizer]
\label{lem:unique_maximizer}
Let $(\Omega, \mathcal{F}, \mu)$ be a $\sigma$-finite measure space. Let $\mathscr{G}_0$ and $\mathscr{G}_1$ be sets of probability measures $G_0$ and $G_1$ absolutely continuous with respect to $\mu$, with densities $g_0=\mathrm{d}G_0/\mathrm{d}\mu$ and $g_1=\mathrm{d}G_1/\mathrm{d}\mu$ respectively. Assume $\mathscr{G}_0 \times \mathscr{G}_1$ is convex and compact in some topology. For $u \in (0,1)$, recall that
\begin{equation}
D_u(G_0, G_1) = \int_\Omega g_1^u \, g_0^{1-u} \, d\mu.
\end{equation}
If $\mathscr{G}_0 \cap \mathscr{G}_1 = \emptyset$ and
\begin{equation}
\max_{(G_0,G_1) \in \mathscr{G}_0 \times \mathscr{G}_1} D_u(G_0, G_1) > 0,
\end{equation}
then $D_u$ admits a unique maximizer $(G_0^*, G_1^*)$ in $\mathscr{G}_0 \times \mathscr{G}_1$, up to $\mu$-null sets.
\end{lem}

\begin{IEEEproof}
See Appendix~\ref{appendix1}.
\end{IEEEproof}


\begin{kor}[Uniqueness transfer from AMR to FMR]\label{cor:unique_transfer}
If for the minimizing parameter $u^*$ the maximizer of $D_{u^*}$ is unique, then any existing FMR LFDs must coincide $\mu$-a.e.\ with the AMR LFDs.  
Thus AMR uniqueness implies uniqueness of all finite-sample LFDs whenever they exist.
\end{kor}

\begin{IEEEproof}
Immediate from Theorem~\ref{thm:AMR_implies_FMR_conditional}.
\end{IEEEproof}

\subsection{Poblem Formulation}\label{optimization}
The design of asymptotically minimax robust tests is based on the optimization problem
\begin{equation}
\min_{u\in(0,1)} \; \max_{(G_0,G_1)\in\mathscr{G}_0\times \mathscr{G}_1} D_u(G_0,G_1).
\end{equation}
As shown in our earlier work~\cite{gularxiv}, this problem admits a saddle value under mild assumptions on the uncertainty classes, with associated least favorable distributions $(\hat{G}_0,\hat{G}_1)$ and minimizing parameter $\hat{u}$. The precise saddle-point characterization and existence conditions are given in~\cite{gularxiv}.\\
Given the existence of $(\hat{G}_0,\hat{G}_1,\hat{u})$, the least favorable distributions can be obtained by solving the coupled minimax optimization problem,
\begin{equation}\label{eq47}
    \begin{aligned}[b]
        \text{Maximization:}\quad &
        \begin{aligned}[t]
            &\hat{g}_0=\mathrm{arg}\sup_{G_0\in\mathscr{G}_0}D_{u}(G_0,G_1)
             \quad \text{s.t. $g_0>0$, $\Upsilon(G_0)=\int_{\Omega}g_0\, d\mu=1$}\\
            &\hat{g}_1=\mathrm{arg}\sup_{G_1\in\mathscr{G}_1}D_{u}(G_0,G_1)
             \quad \text{s.t. $g_1>0$, $\Upsilon(G_1)=\int_{\Omega}g_1\,d\mu=1$}
        \end{aligned}
        \\[12pt]
        \text{Minimization:}\quad &\hat u=\mathrm{arg}\min_{u\in  (0,1)}D_{u}(\hat{G}_0,\hat{G}_1).
    \end{aligned}
\end{equation}
\section{Finite-Sample Minimax Robust Tests via Asymptotic Theory}\label{sec3}
In this section, LFDs and the asymptotically minimax robust tests are derived for various uncertainty classes considering the minimax optimization problem given by \eqref{eq47}. The analysis includes the uncertainty classes based on the total variation distance as well as the band model. The derivations are intentionally presented in full generality so that the same analytic strategy can be applied by other researchers to new uncertainty classes. To keep the exposition focused, detailed proofs are placed in the appendix. By analytically solving the coupled Lagrangian optimality conditions, classical results are recovered, new ones are obtained (e.g., asymmetric total variation neighborhoods and one specialization of the band model), and the structural consequences of the theory are revealed — in particular, the clipped, piecewise-defined form and uniqueness of the robust likelihood ratio, as well as the fact that its parameters are independent of the choice of $u$.\\ 

\subsection{Total Variation Neighborhood}
The total variation neighborhood is defined as
\begin{equation*}
{\mathscr{G}}_j=\{G_j:D_{\mathrm{TV}}(G_j,F_j)\leq \epsilon_j\},\quad j\in\{0,1\},
\end{equation*}
where
\begin{equation*}
D_{\mathrm{TV}}(G_j,F_j)=\frac{1}{2}\int_{\Omega}|g_j-f_j| d\mu.
\end{equation*}
The LFDs and the corresponding minimax robust test for the uncertainty classes created by the total variation neighborhood were found earlier by Huber \cite{hube65}. However, the design approach is heuristic, many choices of the parameters and/or functions are unknown and the test is obtained under the assumption that the robustness parameters are equal $\epsilon_0=\epsilon_1$. Since asymptotic minimax robustness is a necessary condition for finite-sample minimax robustness, the minimax robust test resulting from the total variation neighborhood can also be analytically derived following the same design procedure as before. The following theorem substantiate this claim.

\begin{thm}[Least Favorable Distributions and Robust LRF under Total Variation]\label{theorem31}
Under total variation neighborhoods with radii $\epsilon_0$ and $\epsilon_1$, the robust likelihood ratio function is given by
\begin{equation}
\hat l =
\begin{cases}
t_l, & l < t_l,\\[0.3em]
l, & t_l \le l \le t_u,\\[0.3em]
t_u, & l > t_u,
\end{cases}
\end{equation}
where $0<t_l<t_u<\infty$. The least favorable distributions corresponding to the robust LRF take the form
\begin{align}
\hat g_0 =
\begin{cases}
(1-\beta t_l) f_0 + \beta f_1, & l < t_l,\\[0.3em]
f_0, & t_l \le l \le t_u,\\[0.3em]
(1-\sigma t_u) f_0 + \sigma f_1, & l > t_u,
\end{cases}\quad
\hat g_1 =
\begin{cases}
t_l \hat g_0, & l < t_l,\\[0.3em]
f_1, & t_l \le l \le t_u,\\[0.3em]
t_u \hat g_0, & l > t_u.
\end{cases}
\end{align}
where the coefficients
\begin{equation}
\beta(t_l,t_u)
=
\frac{\epsilon_0}
{\displaystyle \int_{l<t_l} (t_l f_0-f_1)\, d\mu},
\qquad
\sigma(t_l,t_u)
=
\frac{\epsilon_0}
{\displaystyle \int_{l>t_u} (f_1-t_u f_0)\, d\mu}
\end{equation}
are explicit functions of the clipping thresholds $t_l$ and $t_u$.
\end{thm}

\begin{proof}
Lemma~\ref{lemma1} establishes that, under total variation neighborhoods, the robust likelihood ratio function admits the clipped form with thresholds $0<t_l<t_u<\infty$. Given this robust likelihood ratio, Lemma~\ref{lemma2} shows that the corresponding least favorable distributions must coincide with the nominal densities on the central region $\{t_l\le l\le t_u\}$ and are linear combinations of $f_0$ and $f_1$ on the lower and upper clipping regions, respectively. Imposing continuity at the region boundaries yields the explicit parametric expressions stated in the theorem. The remaining dependence of the coefficients $\beta$ and $\sigma$ on the clipping thresholds $t_l$ and $t_u$ is determined by the normalization and total variation constraints, and is made explicit in Proposition~\ref{lem:clip_equation}. 
\end{proof}

\begin{lem}[Parametric form of the robust LRF]\label{lemma1}
For the total variation neighborhod, the robust likelihood ratio function admits a piecewise parametric representation of the form
\begin{equation}
\frac{\hat g_1}{\hat g_0}=
\begin{cases}
t_1, & \frac{f_1}{f_0}<t_l,\\[0.3em]
\frac{f_1}{f_0}, & t_l\le \frac{f_1}{f_0}\le t_u,\\[0.3em]
t_u, & \frac{f_1}{f_0}>t_u,
\end{cases}
\end{equation}
where $t_l<t_u$ are determined by the Lagrange multipliers. Moreover, the parameter $u \in (0,1)$ is uniquely determined by the Lagrange multipliers as
\begin{equation}\label{eq_u}
u=\frac{\log\!\left(\frac{\lambda_0+2\mu_0}{-\lambda_0+2\mu_0}\right)}{\log\!\left(\frac{\lambda_0+2\mu_0}{-\lambda_0+2\mu_0}\right)+\log\!\left(\frac{-\lambda_1+2\mu_1}{\lambda_1+2\mu_1}\right)}.
\end{equation}
\end{lem}

\begin{proof}
The proof proceeds by analyzing the pointwise optimality conditions of the Lagrangian formulation associated with the total variation constraints. For fixed $u\in(0,1)$, consider the Lagrangians
\begin{align}
L_0(g_0,g_1;\lambda_0,\mu_0)&=D_u(G_0,G_1)+\lambda_0 (D_{\mathrm{TV}}(G_0,F_0)-\epsilon_0)+\mu_0(\Upsilon(G_0)-1)),\nonumber\\
L_1(g_0,g_1;\lambda_1,\mu_1)&=D_u(G_0,G_1)+\lambda_1 (D_{\mathrm{TV}}(G_1,F_1)-\epsilon_1)+\mu_1(\Upsilon(G_1)-1)),
\end{align}
with Lagrange multipliers $\lambda_j\ge0$ and $\mu_j\in\mathbb{R}$. Taking pointwise first variations with respect to $g_0$ and $g_1$ yields the
first-order stationarity conditions for $\mu$-almost every $y\in\Omega$,
\begin{align}
(1-u)\!\left(\frac{g_1}{g_0}\right)^u
+\mu_0
+\frac{\lambda_0}{2}\operatorname{sgn}(g_0-f_0)
&=0,
\label{eq:stat_TV_0}\\
u\!\left(\frac{g_1}{g_0}\right)^{u-1}
+\mu_1
+\frac{\lambda_1}{2}\operatorname{sgn}(g_1-f_1)
&=0.
\label{eq:stat_TV_1}
\end{align}
Equations \eqref{eq:stat_TV_0}–\eqref{eq:stat_TV_1} show that the ratio $g_1/g_0$ can take only finitely many constant values, depending on the signs of $g_0-f_0$ and $g_1-f_1$. There are nine possible combinations:
\begin{equation}
g_0 \lessgtr f_0,\qquad g_1 \lessgtr f_1,\qquad g_0=f_0,\ g_1=f_1.
\end{equation}
These cases are examined next.\\
\emph{Case 1: $g_0=f_0$ and $g_1=f_1$.}
Since both constraints are inactive, the optimality conditions imply
\begin{equation}
\frac{g_1}{g_0}=\frac{f_1}{f_0},
\end{equation}
which defines the central region.\\
\emph{Case 2: $g_0<f_0$ and $g_1>f_1$.}
We have $\operatorname{sgn}(g_0-f_0)=-1$ and $\operatorname{sgn}(g_1-f_1)=+1$. Solving \eqref{eq:stat_TV_0}–\eqref{eq:stat_TV_1} gives
\begin{equation}
\frac{g_1}{g_0}=\frac{u(-\lambda_0+2\mu_0)}{(1-u)(\lambda_1+2\mu_1)}=:t_l.
\end{equation}
Moreover, since $g_0<f_0$ and $g_1>f_1$, dividing both inequalities pointwise yields $\frac{g_1}{g_0}>\frac{f_1}{f_0}$\\
\emph{Case 3: $g_0>f_0$ and $g_1<f_1$.}
Here $\operatorname{sgn}(g_0-f_0)=+1$ and $\operatorname{sgn}(g_1-f_1)=-1$, and similarly
\begin{equation}
\frac{g_1}{g_0}=\frac{u(2\mu_0+\lambda_0)}{(1-u)(2\mu_1-\lambda_1)}=:t_u,
\end{equation}
where in this case we also have $\frac{g_1}{g_0}<\frac{f_1}{f_0}$.

\emph{Remaining cases.}
All other sign combinations either reduce to one of the above three cases by dividing the defining inequalities by $g_0$ or $f_0$, or lead to infeasible solutions (e.g., negative constants for $g_1/g_0$), which are ruled out by nonnegativity of densities. Thus, no additional values of $g_1/g_0$ arise. Collecting the admissible cases, the likelihood ratio function $\hat g_1/\hat g_0$ admits the piecewise representation as given in Lemma~\ref{lemma1}.\\ 
The four equations—a pair from each case—must be compatible with a single $u \in (0,1)$. Eliminating $t_l$ and $t_u$ between the two pairs yields a consistency condition that determines $u$ uniquely. After straightforward algebra one obtains $u$ as given in \eqref{eq_u}. This completes the proof.
\end{proof}

\begin{lem}\label{lemma2}
Suppose the robust LRF $\hat l$ admits the clipped form characterized in Lemma~\ref{lemma1}, with thresholds $0<t_l<t_u<\infty$ and a central region on which $\hat l=f_1/f_0$. Then there exist densities $(\hat g_0,\hat g_1)$ such that
\begin{itemize}
\item $\hat g_0=f_0$ and $\hat g_1=f_1$ on the region $\{t_l\le f_1/f_0\le t_u\}$,
\item $\hat g_0$ and $\hat g_1$ are linear combinations of $f_0$ and $f_1$ on the
clipping regions $\{f_1/f_0<t_l\}$ and $\{f_1/f_0>t_u\}$,
\end{itemize}
and the resulting pair $(\hat g_0,\hat g_1)$ is given by the expressions stated
in Theorem~\ref{theorem31}.
\end{lem}

\begin{proof}
Suppose the robust likelihood ratio function is given by Lemma~\ref{lemma1}. On a region where $\hat g_1/\hat g_0$ is a constant the pointwise stationary conditions imply that both $\hat g_0$ and $\hat g_1$ belong to the linear span of the nominal densities $f_0$ and $f_1$. Consequently, on the lower clipping region $\{f_1/f_0<t_l\}$ there exist parameters $\alpha,\beta$ such that
\begin{equation}
\hat g_0 = \alpha f_0 + \beta f_1,
\qquad
\hat g_1 = t_l(\alpha f_0 + \beta f_1),
\label{eq:LFD_lower_general}
\end{equation}
and on the upper clipping region $\{f_1/f_0>t_u\}$ there exist parameters
$\gamma,\sigma$ such that
\begin{equation}
\hat g_0 = \gamma f_0 + \sigma f_1,
\qquad
\hat g_1 = t_u(\gamma f_0 + \sigma f_1).
\label{eq:LFD_upper_general}
\end{equation}
On the middle region $\{t_l\le f_1/f_0\le t_u\}$, the likelihood ratio coincides with the nominal likelihood ratio, and the pointwise stationarity conditions are satisfied by the nominal densities, yielding
\begin{equation}
\hat g_0 = f_0, \qquad \hat g_1 = f_1.
\label{eq:LFD_middle}
\end{equation}
It remains to determine the parameters in these linear representations. The parameters in \eqref{eq:LFD_lower_general} and \eqref{eq:LFD_upper_general} are further restricted by continuity of $\hat g_0$ and $\hat g_1$ at the boundaries of the three regions. Let $y_l$ be such that $f_1(y_l)/f_0(y_l)=t_l$. Continuity at $y_l$ requires
\begin{equation}
\alpha f_0(y_l) + \beta f_1(y_l) = f_0(y_l),
\end{equation}
which implies $\alpha+\beta t_l=1$. Similarly, letting $y_u$ satisfy $f_1(y_u)/f_0(y_u)=t_u$, continuity at $y_u$ yields $\gamma+\sigma t_u=1$. Substituting these relations into \eqref{eq:LFD_lower_general}--\eqref{eq:LFD_upper_general} together with \eqref{eq:LFD_middle} results in the piecewise expressions stated in Theorem~\ref{theorem31}. This completes the proof.
\end{proof}
\noindent Since the least favorable distributions constructed above are independent of the parameter $u$, they maximize $D_u(G_0,G_1)$ simultaneously for all $u\in(0,1)$.

\begin{prop}[Identification of the Clipping Thresholds]\label{lem:clip_equation}
Let $\hat g_0$ and $\hat g_1$ be the least favorable distributions under total variation neighborhoods with radii $\epsilon_0$ and $\epsilon_1$, given in Theorem~\ref{theorem31}.  
Then the clipping thresholds $t_l$ and $t_u$ are uniquely determined by the system
\begin{align}
\int_{l < t_l} \bigl( t_l f_0 - f_1 \bigr)\, d\mu - \epsilon_0 \, t_l &= \epsilon_1, \nonumber\\
\int_{l > t_u} \bigl( f_1 - t_u f_0 \bigr)\, d\mu - \epsilon_0 \, t_u &= \epsilon_1. \label{eq:tu_equation}
\end{align}
\end{prop}

\begin{proof}
The least favorable distributions $\hat g_0$ and $\hat g_1$ are subject to the four defining constraints
\begin{align}
D_{\mathrm{TV}}(G_0,F_0)=\epsilon_0,\quad D_{\mathrm{TV}}(G_1,F_1)=\epsilon_1,\quad \Upsilon(G_0)=1, \quad \Upsilon(G_1)=1.
\end{align}
Substituting the parametric forms of $\hat g_0$ and $\hat g_1$ given in Theorem~\ref{theorem31} into these constraints yields four scalar equations involving the unknown parameters $\beta$, $\sigma$, $t_l$, and $t_u$. The unit-mass constraint for $\hat g_0$ together with the total variation constraint for $\hat g_0$ uniquely determine the coefficients $\beta$ and $\sigma$ as functions of the clipping thresholds $t_l$ and $t_u$. In particular, these coefficients can be eliminated from the remaining constraints by expressing them in terms of $t_l$ and $t_u$. Inserting these expressions into the unit-mass constraint for $\hat g_1$ and the total variation constraint for $\hat g_1$, and simplifying using the piecewise structure induced by the clipping regions $\{l < t_l\}$ and $\{l > t_u\}$, the resulting conditions reduce to the pair of equations given in \ref{eq:tu_equation}.
\end{proof}

\begin{rem}[Symmetric total variation neighborhoods]
When $\epsilon_0=\epsilon_1$, the minimax problem is invariant under interchange of the hypotheses $f_0$ and $f_1$. Since the least favorable likelihood ratio is unique, its clipped form must be invariant under the transformation $l \mapsto 1/l$. This invariance maps the lower clipping region $\{l<t_l\}$ onto the upper clipping region $\{l>t_u\}$, which is possible if and only if $t_l t_u = 1$. Consequently, the symmetric total variation case reduces to a one-parameter family of least favorable distributions. Under this symmetry constraint, the coefficients in the least favorable distributions satisfy $1-\beta t_l = \beta$ and $1-\sigma t_u = \sigma$. Substituting these expressions into the general parametric form recovers Huber’s classical results \cite{hube65}. The present formulation makes explicit that, in the symmetric case, the apparent two-parameter representation of Huber reduces intrinsically to a single parameter.
\end{rem}

\begin{rem}[Tilting structure of the least favorable distributions]
The least favorable distributions admit a simple tilting interpretation. In the clipping regions, they can be written as linear perturbations of the nominal distributions. In particular,
\begin{equation}
\hat g_0
=
f_0
+\epsilon_0
\begin{cases}
\dfrac{f_1 - t_l f_0}{\int_{l<t_l} (t_l f_0-f_1)\, d\mu}, & l < t_l,\\[0.8em]
0, & t_l \le l \le t_u,\\[0.8em]
\dfrac{f_1 - t_u f_0}{\int_{l>t_u} (f_1-t_u f_0)\, d\mu}, & l > t_u,
\end{cases}
\end{equation}
with an analogous expression for $\hat g_1$ through the relation $\hat g_1 = \hat l\,\hat g_0$. Since $f_1 - t_l f_0 < 0$ on $\{l<t_l\}$ and $f_1 - t_u f_0 > 0$ on $\{l>t_u\}$, the perturbation reduces probability mass in the lower likelihood-ratio region and reallocates the same amount of mass to the upper likelihood-ratio region, while leaving the central region $t_l \le l \le t_u$ unchanged. This tilting structure makes explicit how total variation robustness redistributes probability mass in an adversarial yet controlled manner, concentrating it on observations that are most favorable to the competing hypothesis.
\end{rem}

\subsection{Band Model}\label{sec6_band}
So far, the nominal distributions have been assumed to be known or to be reasonably well approximated prior to the construction of the uncertainty classes. However, in many settings such precise knowledge is unavailable, and the true distributions are only known to lie within prescribed pointwise bounds. This motivates the use of band-type uncertainty models \cite{kassamband}, which capture distributional uncertainty through lower and upper bounding functions on the densities.\\
The band model is given by the uncertainty classes
\begin{equation}\label{eq84}
\mathscr{G}_j=\left\{G_j\in\mathscr{M}: g_{j}^L \leq g_j  \leq g_{j}^U \right\}
\end{equation}
where $\mathscr{M}$ is the set of all distribution functions on $\Omega$, and $g_{j}^L$ and $g_{j}^U$ are non-negative lower and upper bounding functions such that $\mathscr{G}_0$ and $\mathscr{G}_1$ are nonempty sets. This implies
\begin{equation*}
\int_{\Omega} g_{j}^L d\mu \leq 1\leq \int_{\Omega} g_{j}^U d\mu,\quad j\in\{0,1\}.
\end{equation*}
Moreover, $g_{j}^L$ and $g_{j}^U$ should be chosen such that $g_0$ and $g_1$ are distinct density functions, if not $\mathscr{G}_0\cap\mathscr{G}_1\neq \emptyset$ and minimax hypothesis testing is not possible.\\
Band models differ fundamentally from distance-based uncertainty classes. While total variation and $f$-divergence neighborhoods constrain distributions only in an integral sense and therefore permit highly concentrated local deviations, band models enforce pointwise bounds on the densities and exclude such behavior by construction. Consequently, band models are not equivalent to, nor recoverable from, distance-based uncertainty classes except in degenerate cases.\\
From a theoretical standpoint, band models constitute capacity-type uncertainty classes. However, it has long been unclear whether they satisfy the structural properties required for the direct application of Huber’s minimax robustness theory, such as alternation \cite{vastola}. As a result, general minimax robustness guarantees could not be established for band models within Huber’s framework \cite{hube73}.\\
From a practical perspective, band models naturally arise in applications where density estimates are accompanied by pointwise confidence bounds, making lower and upper bounding functions an appropriate representation of uncertainty \cite{kassamband}.\\
Following the same  approach as in the previous sections, the asymptotically minimax robust test and least favorable distributions for the band model can be derived as follows. Consider the Lagrangians:
\begin{align*}
L_0(g_0,g_1,\lambda_0,\theta_0,\mu_0)=D_{u}(G_0,G_1)+\lambda_0(g_0-g_0^L)+\nu_0(g_0^U-g_0)+\mu_0(\Upsilon(G_0)-1),\nonumber\\
L_1(g_0,g_1,\lambda_1,\theta_1,\mu_1)=D_{u}(G_0,G_1)+\lambda_1(g_1-g_1^L)+\nu_1(g_1^U-g_1)+\mu_1(\Upsilon(G_1)-1),
\end{align*}
where $\mu_j$ are scalar, and $\lambda_j$ and $\nu_j$ are functional Langrangian multipliers. Taking the Gateaux derivatives of the Lagrangians, at the direction of unit area integrable functions $\psi_0$ and $\psi_1$, respectively, leads to the first-order stationarity conditions
\begin{align}\label{eq87}
\frac{\partial L_0}{\partial g_0}=\int\left((1-u)\left(\frac{g_1}{g_0}\right)^u+\lambda_0-\nu_0+\mu_0\right)\psi_0 d\mu=0,\nonumber\\
\frac{\partial L_1}{\partial g_1}=\int\left(u\left(\frac{g_1}{g_0}\right)^{u-1}+\lambda_1-\nu_1+\mu_1\right)\psi_1 d\mu=0.
\end{align}
The conditions in \eqref{eq87} are accompanied by complementary slackness for the band constraints, which for $j\in\{0,1\}$ read
\begin{align}
\lambda_j(y)\,\big(g_j(y)-g_j^L(y)\big) &= 0, \nonumber\\ 
\nu_j(y)\,\big(g_j^U(y)-g_j(y)\big) &= 0,
\end{align}
with $\lambda_j(y)\ge 0$ and $\nu_j(y)\ge 0$ almost everywhere. Consequently, on regions where a bound is inactive, the corresponding Lagrange multiplier vanishes almost everywhere, and the stationarity conditions reduce to pointwise equations involving only the remaining multipliers. Depending on which bounds are active, three distinct cases arise, which are analyzed separately below.
\subsubsection*{\text{\bf{Case 1}.} $g_0^U=\infty$ and $g_1^U=\infty$ \text{(no upper bounding functions)}}
\noindent In this case, letting $g_{j}^L=(1-\epsilon_j)f_j$, the band model can equivalently be written as the lower $\epsilon$-contamination model
\begin{equation*}
\mathscr{G}_j^{\epsilon^{-}}=\left\{G_j: G_j=(1-\epsilon_j)F_j+\epsilon_j H, H\in\mathscr{M} \right\}
\end{equation*}
where $f_j$ are the nominal densities and $0\leq\epsilon_j<1$ \cite{kassamband}. Since we have $\nu_0=0$ and $\nu_1=0$ everywhere, and hence, no constraints regarding the upper bounding functions are in effect, there are four conditions regarding the Lagrangians
\begin{align*}
L_0:&\quad g_0=g_0^L\quad \mbox{on}\quad A_0\quad \mbox{and}\quad g_0>g_0^L\quad \mbox{on}\quad \Omega\backslash A_0,\nonumber\\
L_1:&\quad g_1=g_1^L\quad \mbox{on}\quad A_1\quad \mbox{and}\quad g_1>g_1^L\quad \mbox{on}\quad \Omega\backslash A_1.
\end{align*}
The integrals in \eqref{eq87} are defined on the regions where $g_0>g_0^L$ and $g_1>g_1^L$, respectively. On these regions, the lower band constraints are inactive and, by complementary slackness, the corresponding multipliers satisfy $\lambda_0=\lambda_1=0$ almost everywhere. Since $\nu_0=0$ and $\nu_1=0$ everywhere in this case, the stationarity conditions reduce to 
\begin{align}\label{eq89}
\frac{g_1}{g_0}=\frac{1}{k_2}\quad \mbox{on}\quad \bar{A}_0=\Omega\backslash A_0=\{y:g_0>g_0^L\},\nonumber\\
\frac{g_1}{g_0}=k_1\quad \mbox{on}\quad \bar{A}_1=\Omega\backslash A_1=\{y:g_1>g_1^L\},
\end{align}
where $k_1$ and $k_2$ are positive constants determined by the normalization constraints.
\begin{thm}\label{theorem1}
From \eqref{eq89}, it follows that the LFDs and the corresponding likelihood ratio function are unique and given by
\begin{equation}\label{eq90}
\hat{g}_0=\begin{cases}
g_0^L, &  y\in A_0 \\
k_2g_1^L, &  y\in \bar{A}_0
\end{cases},\quad
\hat{g}_1=\begin{cases}
g_1^L, &  y\in A_1 \\
k_1g_0^L, &  y\in \bar{A}_1
\end{cases},
\end{equation}
and
\begin{equation*}
\frac{\hat{g}_1}{\hat{g}_0}=\begin{cases}
\frac{1}{k_2}, &  y\in \bar{A}_0\cap A_1 \\
\frac{g_1^L}{g_0^L}, &  y\in A_0\cap A_1\\
k_1, &  y\in A_0\cap \bar{A}_1
\end{cases}.
\end{equation*}
\end{thm}
\begin{IEEEproof}
The claim follows from the conditions:
\begin{enumerate}
\renewcommand\labelenumi{\bfseries\theenumi.}
\item The sets $A_0$, $A_1$, $\bar{A}_0$ and $\bar{A}_1$ are all non-empty.
\item The set $\bar{A}_0\cap \bar{A}_1$ is empty.
\item On $\bar{A}_0$ and $\bar{A}_1$, respectively, we have $\hat{g}_0=k_2g_1^L$ and $\hat{g}_1=k_1g_0^L$.
\end{enumerate}
\end{IEEEproof}
A detailed proof for each of these conditions is given in Appendix~\ref{appendix_band1}.

\begin{kor}\label{corollary2}
The parameters should satisfy $k_1<1/k_2$, hence,
\begin{equation*}
A_0\cap A_1=\{k_1\leq g_1^L/g_0^L \leq 1/k_2\}.
\end{equation*}
Moreover,
\begin{equation*}
A_0=\{g_1^L/g_0^L<1/k_2\},\quad A_1=\{g_1^L/g_0^L>k_1\}.
\end{equation*}
\end{kor}

\begin{IEEEproof}
A proof of Corollary~\ref{corollary2} is given in Appendix~\ref{appendix_cor2}
\end{IEEEproof}

\begin{rem}\label{rem_contamination}
Let $t_u=1/k_2$, $t_l=k_1$ and $l=g_1^L/g_0^L$. Then, the LFDs and the robust LRF can be rewritten as
\begin{equation}\label{eq97}
\hat{g}_0=\begin{cases}
g_0^L, &  l\leq t_u \\
1/t_ug_1^L, &  l> t_u
\end{cases},\quad
\hat{g}_1=\begin{cases}
g_1^L, &  l\geq t_l \\
t_lg_0^L, &  l< t_l
\end{cases},
\end{equation}
and
\begin{equation}\label{eq98}
\frac{\hat{g}_1}{\hat{g}_0}=\begin{cases}
t_u, &  l>t_u \\
l, &  t_l\leq l \leq t_u\\
t_l, &  l<t_l
\end{cases}.
\end{equation}
The lower bounding function constraints are satisfied automatically. Because, on $\{l\leq t_u\}$ and $\{l\geq t_l\}$, $\hat{g}_j\geq g_j^L$ holds with equality, and on $\{l> t_u\}$ and $\{l< t_l\}$, we necessarily have $\hat{g}_0=1/t_ug_1^L\geq g_0^L$ and $\hat{g}_1= t_lg_0^L\geq g_1^L$, respectively, as $l=g_1^L/g_0^L$. The density function constraints are satisfied by solving
\begin{align}\label{eq99}
\int_{l\leq t_u} &g_0^L d\mu+\frac{1}{t_u}\int_{l> t_u} g_1^L \mathrm{d}\mu=1, \nonumber\\
\int_{l\geq t_l} &g_1^L d\mu+t_l\int_{l< t_l} g_0^L d\mu=1.
\end{align}
\end{rem}

\begin{rem}
Let $c'$ and $c''$ denote Huber's clipping thresholds associated with the likelihood ratio $\ell=f_1/f_0$. If the band-model clipping thresholds are parametrized as
\begin{equation}
t_l = c'\frac{1-\epsilon_1}{1-\epsilon_0},
\qquad
t_u = c''\frac{1-\epsilon_1}{1-\epsilon_0},
\end{equation}
then the LFDs in \eqref{eq97} reduce exactly to Huber's $\epsilon$-contamination least favorable distributions \cite{hube65}. In this case it is known by Huber that the equations in \eqref{eq99} have unique solutions and the LFDs in \eqref{eq97} are single-sample minimax robust \cite{hube65}. From Theorem~\ref{thm1}, single-sample minimax robust LFDs minimize all $f$-divergences, hence they also maximize all $u$-affinities.
By \eqref{eq99}, it is also implied that the parameters $t_l$ and $t_u$ are only dependent on $g_0^L$ and $g_1^L$, i.e. they are independent of the choice of $u$. This is in accordance with Theorem~\ref{thm1}. 
\end{rem}

\subsubsection*{\text{\bf{Case 2}.} $g_0^L=0$ and $g_1^L=0$ \text{(no lower bounding functions)}}
\noindent In this case, letting $g_{j}^U=(1+\epsilon_j)f_j$, the band model can equivalently be written as the upper $\epsilon$-contamination model
\begin{equation*}
\mathscr{G}_j^{\epsilon^{+}}=\left\{G_j: G_j=(1+\epsilon_j)F_j-\epsilon_j H, H\in\mathscr{M} \right\}
\end{equation*}
where $f_j$ are the nominal density functions and $\epsilon_j>0$. By the condition of no lower bounding functions, we have $\lambda_0=0$ and $\lambda_1=0$ everywhere. Similarly, the positivity constraints are also not imposed as before because, as it can be seen later, the density functions automatically satisfy these constraints. In this case, there are four conditions regarding the Lagrangians:
\begin{align}\label{eq100}
L_0:&\quad g_0=g_0^U\quad \mbox{on}\quad A_0\quad \mbox{and}\quad g_0<g_0^U\quad \mbox{on}\quad \Omega\backslash A_0,\nonumber\\
L_1:&\quad g_1=g_1^U\quad \mbox{on}\quad A_1\quad \mbox{and}\quad g_1<g_1^U\quad \mbox{on}\quad \Omega\backslash A_1.
\end{align}
The integrals in \eqref{eq87} are defined on the regions where $g_0<g_0^U$ and $g_1<g_1^U$, respectively. On these regions, the upper band constraints are inactive and, by complementary slackness, the corresponding multipliers satisfy $\nu_0=\nu_1=0$ almost everywhere. Since $\lambda_0=0$ and $\lambda_1=0$ everywhere in this case, the stationarity conditions reduce to
\begin{align}\label{eq101}
\frac{g_1}{g_0}=\frac{1}{k_2}\quad \mbox{on}\quad \bar{A}_0=\Omega\backslash A_0=\{y:g_0<g_0^U\},\nonumber\\
\frac{g_1}{g_0}=k_1\quad \mbox{on}\quad \bar{A}_1=\Omega\backslash A_1=\{y:g_1<g_1^U\},
\end{align}
where $k_1$ and $k_2$ are positive constants determined by the normalization constraints.

\begin{thm}\label{theorem2}
Let $t_l=1/k_2$, $t_u=k_1$ and $l=g_1^U/g_0^U$. It follows that the LFDs and the corresponding LRF are unique and given by
\begin{equation}
\hat{g}_0=\begin{cases}
g_0^U, &  l\geq t_l  \\
1/t_lg_1^U, &  l< t_l
\end{cases},\quad
\hat{g}_1=\begin{cases}
g_1^U, &  l\leq t_u \\
t_ug_0^U, &  l> t_u
\end{cases},
\end{equation}
and
\begin{equation}
\frac{\hat{g}_1}{\hat{g}_0}=\begin{cases}
t_l, &  l<t_l\\
l, &  t_l\leq l \leq t_u\\
t_u, &  l>t_u
\end{cases}.
\end{equation}
Moreover, all the Lagrangian constraints are satisfied and in particular the LFDs are obtained by solving
\begin{align*}
\int_{l\geq t_l} &g_0^U d\mu+\frac{1}{t_l}\int_{l< t_l} g_1^U \mathrm{d}\mu=1, \nonumber\\
\int_{l\leq t_u} &g_1^U d\mu+t_u\int_{l> t_u} g_0^U d\mu=1.
\end{align*}
\end{thm}

\begin{IEEEproof}
A proof of Theorem~\ref{theorem2} is given in Appendix~\ref{appendix_thm2}.
\end{IEEEproof}

\begin{thm}\label{theoremeps}
The LFDs in Theorem~\ref{theorem2} are single-sample minimax robust, i.e.
\begin{align}\label{eq105}
&G_0\left[\hat{l}< t\right]\geq \hat{G}_0\left[\hat{l}< t\right],\nonumber\\
&G_1\left[\hat{l}< t\right]\leq \hat{G}_1\left[\hat{l}< t\right]
\end{align}
for all $t\in\mathbb{R}_{\geq 0}$ and $(G_0,G_1)\in\mathscr{G}_0\times\mathscr{G}_1$.
\end{thm}

\begin{IEEEproof}
A proof of Theorem~\ref{theoremeps} is given in Appendix~\ref{appendix_thmeps}.
\end{IEEEproof}

\subsubsection*{\text{\bf{Case 3}.} $g_j^L<g_j<g_j^U$ \text{(the general case)}}
\noindent The uncertainty classes for the general case are obtained by the intersection of lower and upper $\epsilon$-contamination neighborhoods
\begin{equation*}
\mathscr{G}_j=\mathscr{G}_j^{\epsilon^{-}}\cap\mathscr{G}_j^{\epsilon^{+}}.
\end{equation*}
There are six conditions regarding the Lagrangians
\begin{align}\label{eq108}
L_0:&\quad g_0=g_0^L\quad \mbox{on}\quad A_0,\quad g_0=g_0^U\quad \mbox{on}\quad A_1\quad \mbox{and}\quad g_0^L<g_0<g_0^U \quad \mbox{on}\quad A_2,\nonumber\\
L_1:&\quad g_1=g_1^L\quad \mbox{on}\quad A_3,\quad  g_1=g_1^U\quad \mbox{on}\quad A_4\quad \mbox{and}\quad g_1^L<g_1<g_1^U\quad \mbox{on}\quad A_5.
\end{align}
On regions where $g_0^L < g_0 < g_0^U$ and $g_1^L < g_1 < g_1^U$, the functional multipliers $\lambda_j$ and $\nu_j$ vanish pointwise by complementary slackness, and the stationarity conditions reduce to
\begin{align}\label{eq109}
\frac{g_1}{g_0}=k_2\quad \mbox{on}\quad  A_2=\{y:g_0^L<g_0<g_0^U\},\nonumber\\
\frac{g_1}{g_0}=k_1\quad \mbox{on}\quad  A_5=\{y:g_1^L<g_1<g_1^U\},
\end{align}
where $k_1$ and $k_2$ are some positive constants.

\begin{thm}\label{theorem3}
There are three different asymptotically minimax robust likelihood ratio functions,
\begin{equation*}
\text{\rm Type-A}:\quad\frac{\hat{g}_1}{\hat{g}_0}=\begin{cases}
g_1^U/g_0^L, &  g_1^U/g_0^L\leq k_2\\
k_2, &  g_1^U/g_0^L>k_2>g_1^U/g_0^U\\
g_1^U/g_0^U, &  k_2\leq g_1^U/g_0^U\leq k_1\\
k_1, &  g_1^U/g_0^U>k_1>g_1^L/g_0^U\\
g_1^L/g_0^U, &  g_1^L/g_0^U\geq k_1
\end{cases},
\end{equation*}

\begin{equation*}
\text{\rm Type-B}:\quad\frac{\hat{g}_1}{\hat{g}_0}=\begin{cases}
g_1^U/g_0^L, &  g_1^U/g_0^L\leq k_1\\
k_1, &  g_1^U/g_0^L>k_1>g_1^L/g_0^U\\
g_1^L/g_0^U, &  g_1^L/g_0^U\geq k_1
\end{cases},
\end{equation*}

\begin{equation*}
\text{\rm Type-C}:\quad\frac{\hat{g}_1}{\hat{g}_0}=\begin{cases}
g_1^U/g_0^L, &  g_1^U/g_0^L\leq k_1\\
k_1, &  g_1^U/g_0^L>k_1>g_1^L/g_0^L\\
g_1^L/g_0^L, &  k_1\leq g_1^L/g_0^L\leq k_2\\
k_2, &  g_1^L/g_0^L>k_2>g_1^L/g_0^U\\
g_1^L/g_0^U, &  g_1^L/g_0^U\geq k_2
\end{cases},
\end{equation*}

with the corresponding pairs of LFDs, respectively,

\begin{equation*}
\hat{g}_0=\begin{cases}
g_0^L, &  g_1^U/g_0^L\leq k_2  \\
\frac{1}{k_2}g_1^U, &  g_1^U/g_0^L>k_2>g_1^U/g_0^U\\
g_0^U, &  g_1^U/g_0^U\geq k_2
\end{cases},\quad
\hat{g}_1=\begin{cases}
g_1^L, &  g_1^L/g_0^U\geq k_1 \\
k_1g_0^U, &  g_1^U/g_0^U>k_1>g_1^L/g_0^U\\
g_1^U, & g_1^U/g_0^U\leq k_1
\end{cases},
\end{equation*}

\begin{equation*}
\hat{g}_0=\begin{cases}
g_0^L, &  g_1^U/g_0^L\leq k_1  \\
k_2(g_0^L+h_1), &  g_1^U/g_0^L>k_1\geq g_1^L/g_0^L\\
\frac{k_2}{k_1}(g_1^L+h_2), &  g_1^L/g_0^L>k_1>g_1^L/g_0^U\\
g_0^U, &  g_1^L/g_0^U\geq k_1
\end{cases},\quad
\hat{g}_1=\begin{cases}
g_1^U, &  g_1^U/g_0^L\leq k_1 \\
k_1k_2(g_0^L+h_1), &  g_1^U/g_0^L>k_1\geq g_1^L/g_0^L\\
k_2(g_1^L+h_2), &  g_1^L/g_0^L>k_1>g_1^L/g_0^U\\
g_1^L, & g_1^L/g_0^U\geq k_1
\end{cases},
\end{equation*}

\begin{equation*}
\hat{g}_0=\begin{cases}
g_0^L, &  g_1^L/g_0^L\leq k_2 \\
\frac{1}{k_2} g_1^L, &  g_1^L/g_0^L>k_2>g_1^L/g_0^U\\
g_0^U, &  g_1^L/g_0^U\geq k_2
\end{cases},\quad
\hat{g}_1=\begin{cases}
g_1^L, &  g_1^L/g_0^L\geq k_1 \\
k_1 g_0^L, &  g_1^U/g_0^L>k_1>g_1^L/g_0^L\\
g_1^U, & g_1^U/g_0^L\leq k_1
\end{cases},
\end{equation*}
Moreover, LRFs of \text{Type-A} and \text{Type-C} tend to clipped likelihood ratio functions, e.g., as given by \eqref{eq98} with the corresponding LFDs defined by \eqref{eq97}.
\end{thm}

\begin{IEEEproof}
A proof of Theorem~\ref{theorem3} is given in Appendix~\ref{appendix_band2}.
\end{IEEEproof}
In practice, two situations may arise. If the structural type of the robust likelihood ratio function is known a priori, the parameters $k_1$ and $k_2$ can be determined by enforcing the unit-mass constraints on the least favorable densities. Otherwise, the minimax problem formulated in Section~\ref{optimization} may be solved numerically as a convex optimization problem over the model \eqref{eq84}. In this case, the densities can be discretized by sampling, and numerical integration may be carried out using standard techniques such as the trapezoidal rule.

\section{Numerical Derivation of Minimax Robust Tests via Convex Optimization}\label{sec4}
When analytic characterizations of the least favorable distributions are unavailable or intractable, the asymptotic minimax design problem can be cast as a convex optimization problem after discretizing the domain~$\Omega$. This numerical formulation applies to a broad class of convex uncertainty sets. In this work, moment classes and $p$-point classes are treated as representative examples.\\
Let $\Omega$ be discretized into $n$ grid points $\{x_1,\dots,x_n\}$ with uniform spacing $\Delta x = (x_{\max}-x_{\min})/(n-1)$. Each density $g$ is represented as a nonnegative vector $g = (g_1,\dots,g_n)$ satisfying the normalization constraint
\begin{equation}
\sum_{i=1}^n g_i \,\Delta x = 1.
\end{equation}
The $u$-affinity functional $D_u(g_0,g_1)$ is then approximated using the trapezoidal integration rule as
\begin{equation}
\widehat{D}_u(g_0,g_1) \approx \sum_{i=1}^n g_{1,i}^u g_{0,i}^{1-u} \,\Delta x.
\end{equation}
In the asymptotic minimax robust formulation, the design of LFDs involves solving the minimax problem
\begin{equation}
\min_{u \in (0,1)} \; \max_{(g_0,g_1) \in \mathscr{G}_0\times\mathscr{G}_1} \; \widehat{D}_u(g_0,g_1),
\end{equation}
where $\mathscr{G}_0$ and $\mathscr{G}_1$ are convex sets describing the uncertainty classes. The inner maximization over $(g_0,g_1)$ is convex for fixed $u$, and the outer minimization over $u$ can be performed via a scalar search.\\ 
Based on Theorem~\ref{thm:FMR_implies_AMR}, whenever a finite-sample minimax robust (FMR) test exists for a given uncertainty class, the same least favorable distributions maximize $D_u$ for all $u\in(0,1)$. In such cases, the minimization over $u$ can be skipped entirely, and $u$ can be set to any convenient value (e.g., $u=0.5$). For the moment and $p$-point classes considered here, both the counterexample provided in \cite{veeravalli_counterexample} and our own numerical experiments indicate that they are not finite-sample minimax robust. Nevertheless, asymptotically minimax robust tests do exist for these classes, and their numerical derivation will be presented in the simulations section. 

\subsection{Moment Classes}
Moment classes, originally introduced in~\cite{moment}, model partial information about the distribution through bounds on generalized moments
\begin{equation}
\mathscr{G}_j=\left\{G_j\in\mathscr{M}: a_j^k\leq \mathbb{E}_{G_j}[h_j^k(Y)] \leq b_j^k\right\},\quad k\in\{1,\ldots,K\},
\end{equation}
where $h_j^k$ are real-valued continuous functions, $(a_j^k,b_j^k)$ are given bounds, and $K$ is the number of constraints. The bounds should be chosen so that $\mathscr{G}_0 \cap \mathscr{G}_1 = \emptyset$. Upon discretization, each moment constraint becomes a linear inequality in $g_j$,
\begin{equation}
a_j^k \leq \sum_{i=1}^n h_j^k(x_i) \, g_{j,i} \,\Delta x \leq b_j^k.
\end{equation}
The resulting finite-dimensional minimax program reads
\begin{align*}
\min_{u \in (0,1)} \ \max_{g_0,g_1} \quad & \sum_{i=1}^n g_{1,i}^u \, g_{0,i}^{\,1-u} \,\Delta x \\
\text{s.t.} \quad & \mathbf{1}^\top g_j \,\Delta x = 1,\quad g_j \geq 0,\quad j=0,1,\\
& a_j^k \leq (h_j^k)^\top g_j \,\Delta x \leq b_j^k,\quad k=1,\dots,K.
\end{align*}
This is a convex optimization problem in $(g_0,g_1)$ for every fixed $u$ and can be solved using general-purpose convex solvers; the minimization over $u$ can be performed, for example, via a simple grid search. An example of LFDs for a specific set of moment bounds is given in Section~\ref{sec5}.

\subsection{P-point Classes}
The p-point class models partial information in the form of probability masses assigned to disjoint subsets of $\Omega$. A generalized definition is
\begin{equation}\label{eq:p_point_class}
\mathscr{G}_j=\left\{G_j\in\mathscr{M}: a_j^k \leq G_j(A_j^k)\leq b_j^k\right\},\quad k\in\{1,\ldots,K\},
\end{equation}
where $A_j^k \in \mathscr{A}$ are disjoint measurable subsets and $(a_j^k,b_j^k)$ specify allowable mass intervals. This generalizes the models in~\cite{elsawy,elsawy2}. When $\Omega$ is discretized, each $A_j^k$ becomes a known index set $\mathcal{I}_j^k \subset \{1,\dots,n\}$, and~\eqref{eq:p_point_class} reduces to linear constraints
\begin{equation}
a_j^k \leq \sum_{i\in\mathcal{I}_j^k} g_{j,i} \,\Delta x \leq b_j^k.
\end{equation}
The resulting optimization has the same structure as in the moment-classes. An example of LFDs under $p$-point constraints is provided in Section~\ref{sec5}.

\subsection{Hybrid Models and Generality of the Method}
The same discretization–optimization framework can handle hybrid models that combine different convex constraints, such as simultaneous moment and $p$-point bounds. More generally, any convex uncertainty class whose constraints admit a convex finite-dimensional representation after discretization can be handled by this method. The choice of $n$ and $\Delta x$ determines the trade-off between numerical accuracy and computational cost; typical values of $n\in\{100,200\}$ yield sufficient accuracy for the smooth densities considered here.

\begin{algorithm}[t]
\caption{Numerical LFD Design via Convex Optimization}
\label{alg:lfd_design}
\begin{algorithmic}[1]
\REQUIRE Uncertainty classes $\mathscr{G}_0$, $\mathscr{G}_1$ with convex constraints; 
         domain $\Omega = [x_{\min},x_{\max}]$; 
         number of grid points $n$; 
         optional fixed $u_0 \in (0,1)$ if $u$-minimization is skipped.
\ENSURE Least favorable densities $(\hat{g}_0,\hat{g}_1)$.
\STATE Discretize $\Omega$ into $n$ grid points $\{x_1,\dots,x_n\}$ with spacing $\Delta x$.
\STATE Express the constraints of $\mathscr{G}_0$, $\mathscr{G}_1$ as inequalities in $g_0$, $g_1$.
\IF{$u$-minimization is required}
    \STATE Initialize a search set $\mathcal{U} \subset (0,1)$.
    \FOR{each $u \in \mathcal{U}$}
        \STATE Solve the convex program:
        \[
        \begin{aligned}
        \text{maximize} \quad & \sum_{i=1}^n g_{1,i}^u g_{0,i}^{1-u} \,\Delta x \\
        \text{subject to} \quad & \mathbf{1}^\top g_j \,\Delta x = 1, \quad g_j \geq 0, \quad j=0,1, \\
                                & \text{class constraints for }\mathscr{G}_0,\mathscr{G}_1.
        \end{aligned}
        \]
        \STATE Record the objective value and corresponding $(g_0,g_1)$.
    \ENDFOR
    \STATE Select $u^\ast$ and $(\hat{g}_0,\hat{g}_1)$ achieving the minimum over $u$.
\ELSE
    \STATE Solve the convex program above for $u = u_0$.
\ENDIF
\RETURN $(\hat{g}_0,\hat{g}_1)$.
\end{algorithmic}
\end{algorithm}

\subsection{Summary}
For moment classes, $p$-point classes, and their hybrids, the LFDs are obtained by discretizing the domain and solving a finite-dimensional convex program for each fixed $u$, combined with an outer minimization over $u$ (which can be implemented via line search or bisection). The overall procedure is summarized in Algorithm~\ref{alg:lfd_design}. This framework extends the analytic solutions available for divergence-based classes and the band model, providing a unified computational method for robust hypothesis testing across a wide range of convex uncertainty classes.

\section{Simulations}\label{sec5}
This section illustrates and evaluates the theoretical findings through representative examples. The notation $\mathcal{N}(\mu,\sigma^2)$ denotes a Gaussian distribution with mean $\mu$ and variance $\sigma^2$, and $f_\mathcal{N}$ denotes the corresponding density function. All systems of equations are solved using the damped Newton method~\cite{ralph}, and convex optimization problems are handled by interior-point methods~\cite{potra}.

\subsection{Total Variation and $\epsilon$-contamination Neighborhoods}
In this subsection, the least favorable densities (LFDs) and the corresponding likelihood ratio functions (LRFs) resulting from the total variation neighborhood are illustrated and compared with those of the classical $\epsilon$-contamination model. Unit-variance mean-shifted Gaussian distributions $f_j \sim \mathcal{N}(2j-1,1)$ are considered as the nominal distributions. For various configurations, the parameters of the least favorable densities given in Theorem~\ref{theorem31} and Remark~\ref{rem_contamination} have been obtained by solving the pair of equations given by \eqref{eq:tu_equation} and \eqref{eq99}, respectively. Figure~\ref{fig1} illustrates the resulting LFDs and the corresponding LRFs together with the nominal distributions. From this example the following observations can be made:

\begin{figure}[ttt]
  \centering
  \centerline{\includegraphics[width=16.5cm]{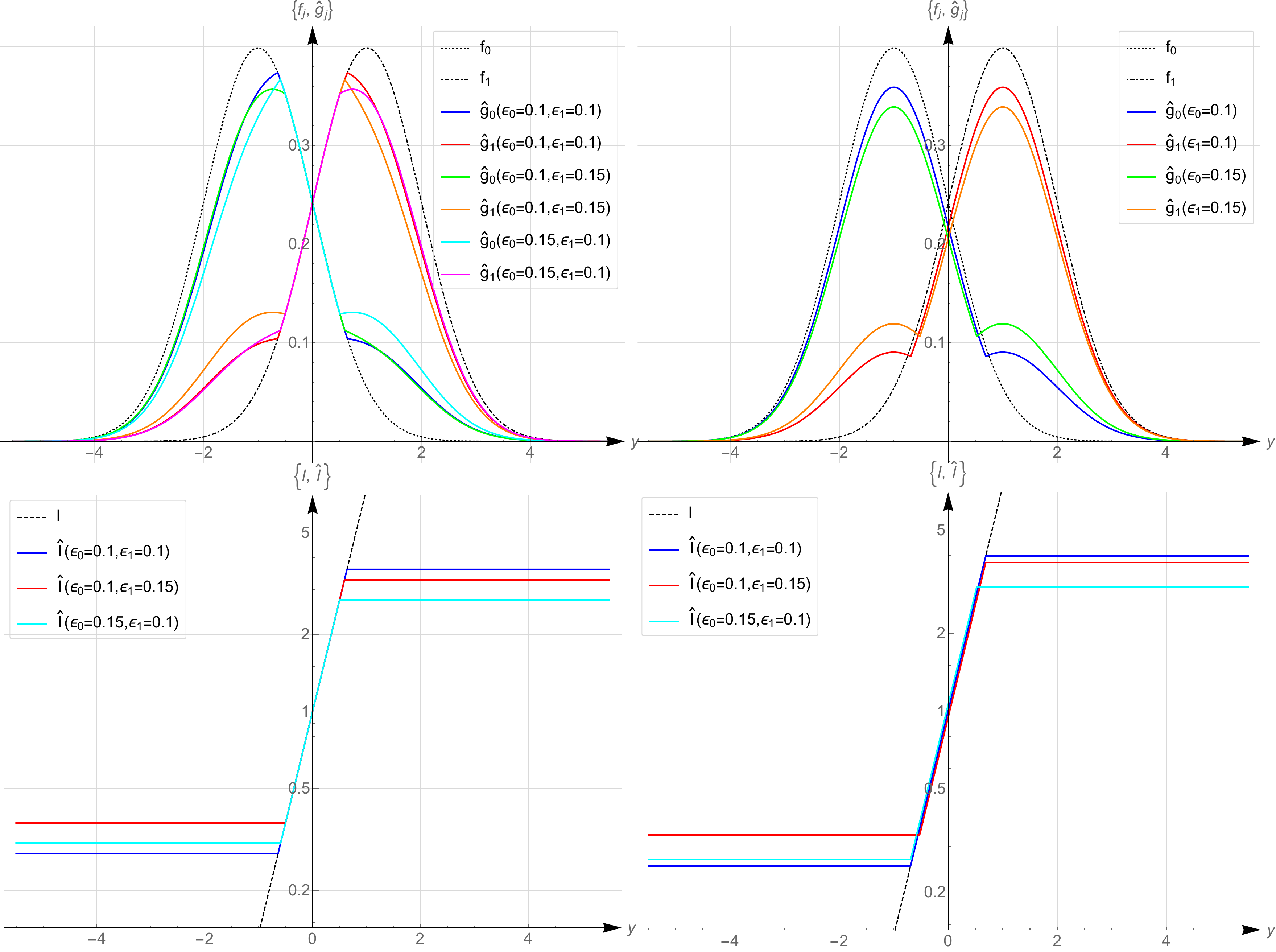}}
\caption{Least favorable densities (top) and robust likelihood ratio functions (bottom) for (left) total variation distance based uncertainty classes and (right) $\epsilon$-contamination model.\label{fig1}}
\end{figure}

\begin{enumerate}
\item Both models lead to clipped likelihood ratio tests; however, their LFDs are not identical even when their LRFs coincide. This can be deduced by examining the region where the robust and nominal LRFs are equal (i.e. for $\epsilon_0=\epsilon_1=0.1$ and $y \in [-0.7,0.7]$). In this region, the LFDs of the total variation neighborhood coincide with each other, while those of the $\epsilon$-contamination model do not. Hence, the two models differ in their least favorable density structures, despite having exactly the same LRFs.
\item Identical clipping thresholds $(t_\ell,t_u)$ can be obtained by appropriately adjusting the robustness parameters. In particular, the total variation model with $\epsilon_0=\epsilon_1\approx 0.08875$ produces the same upper and lower thresholds as the $\epsilon$-contamination model with $\epsilon_0=\epsilon_1=0.1$. This indicates that, in terms of likelihood-ratio compression, the total variation neighborhood induces stronger robustness effects for the same robustnes parameters.
\item The effects of unequal robustness parameters are similar at the lower and upper clipping regions, and different in the middle region. For both models, increasing either $\epsilon_0$ or $\epsilon_1$ decreases the upper clipping threshold and increases the lower one (a vertical compression), while the horizontal displacement—rightward for larger $\epsilon_1$ and leftward for larger $\epsilon_0$—is observed only in the $\epsilon$-contamination model.
\item From a computational standpoint, both models require two equations -each having a single variable- to be solved. While, the two equations of the total variation based uncertainty model are coupled in $\epsilon_0$ and $\epsilon_1$, those that of the $\epsilon$-contamination model are not. This explains why one needs three pairs of LFD for the total variation based uncertainy model and only two pairs for the $\epsilon$-contamination model in Figure~\ref{fig1}.
\item The overall deformation of the LRFs with changing robustness parameters therefore differs fundamentally between the two models.
\end{enumerate}
These comparisons highlight an important distinction: while both uncertainty models yield clipped likelihood ratio tests, their least favorable densities exhibit fundamentally different behaviors, particularly when $\epsilon_0 \neq \epsilon_1$. The new analytical formulation for the total variation neighborhood thus enables robust test design under unequal robustness parameters—a capability not available in previous literature.

\subsection{Band Model}
Asymptotically minimax robust tests arising from the band model can similarly be simulated. Consider the lower bounding functions
\begin{equation*}
g_0^L(y)=(1-\epsilon)f_{\mathcal{N}}(y;-1,4),\quad g_1^L(y)=(1-\epsilon)f_{\mathcal{N}}(y;1,4),
\end{equation*}
where the contamination ratio is chosen to be $\epsilon=0.2$. Furthermore, let the upper bounding functions be
\begin{equation*}
g_0^U(y)=(1+\varepsilon)f_{\mathcal{N}}(y;-1,4),\quad g_1^U(y)=(1+\varepsilon)f_{\mathcal{N}}(y;1,4),
\end{equation*}
with the parameters $\varepsilon=0.2$ (Type-A), $\varepsilon=0.5$ (Type-B), $\varepsilon=1.5$ (Type-C) or $\varepsilon=19$ (Type-C), simulating three different types of robust LRFs resulting from the band model, cf. Section~\ref{sec6_band}.\\
For this setup, and excluding $\varepsilon=19$ for the sake of clarity, Figure~\ref{fig9} illustrates the corresponding LFDs together with the lower bounding functions, and the upper bounding functions for $\varepsilon=0.2$. For $\varepsilon=0.5$, the LFDs are overlapping around $y=0$, leading to $\hat{l}=1$. This type of overlapping has previously been reported by \cite{gul6} for single-sample minimax robust tests obtained from the KL-divergence neighborhood. However, the test in \cite{gul6} is not minimax robust unless a well defined randomized decision rule is used.\\
In Figure~\ref{fig10}, the corresponding robust likelihood ratio functions are illustrated. Increasing $\varepsilon$ transforms the corresponding robust LRF from Type-A to Type-B and then to Type-C. Further increasing $\varepsilon$, i.e. when $\varepsilon=19$, the robust LRF tends to a clipped likelihood ratio test, which is the limiting LRF stated in Section~\ref{sec6_band}. The robust LRFs can take different shapes depending on the bounding functions. Similar patterns were stated in \cite{kassamband} and also observed in \cite{fauss}.\\
In the second example the variance of Type-B upper bounding functions are changed from $4$ to $9$ and the rest of the setup is kept the same as before. Figure~\ref{fig88} illustrates the LFDs together with the upper and lower bounding functions, and the corresponding robust LRF. This example shows both an asymmetric and a degenerate case, where the latter is due to the fact that the region $g_1^U/g_0^L\leq k_1$ does not exist, as the optimum solution satisfies $g_1^U/g_0^L > k_1$. Also the LFDs in the middle two regions, where the corresponding LRF is constant, were determined as a result of numerical optimization, see Section~\ref{sec4}, as the LFDs in these regions depend on two general functions $h_1$ and $h_2$, which cannot be recovered as linear combinations of upper and lower bounding functions.

\begin{figure}[ttt]
  \centering
  \centerline{\includegraphics[width=8.8cm]{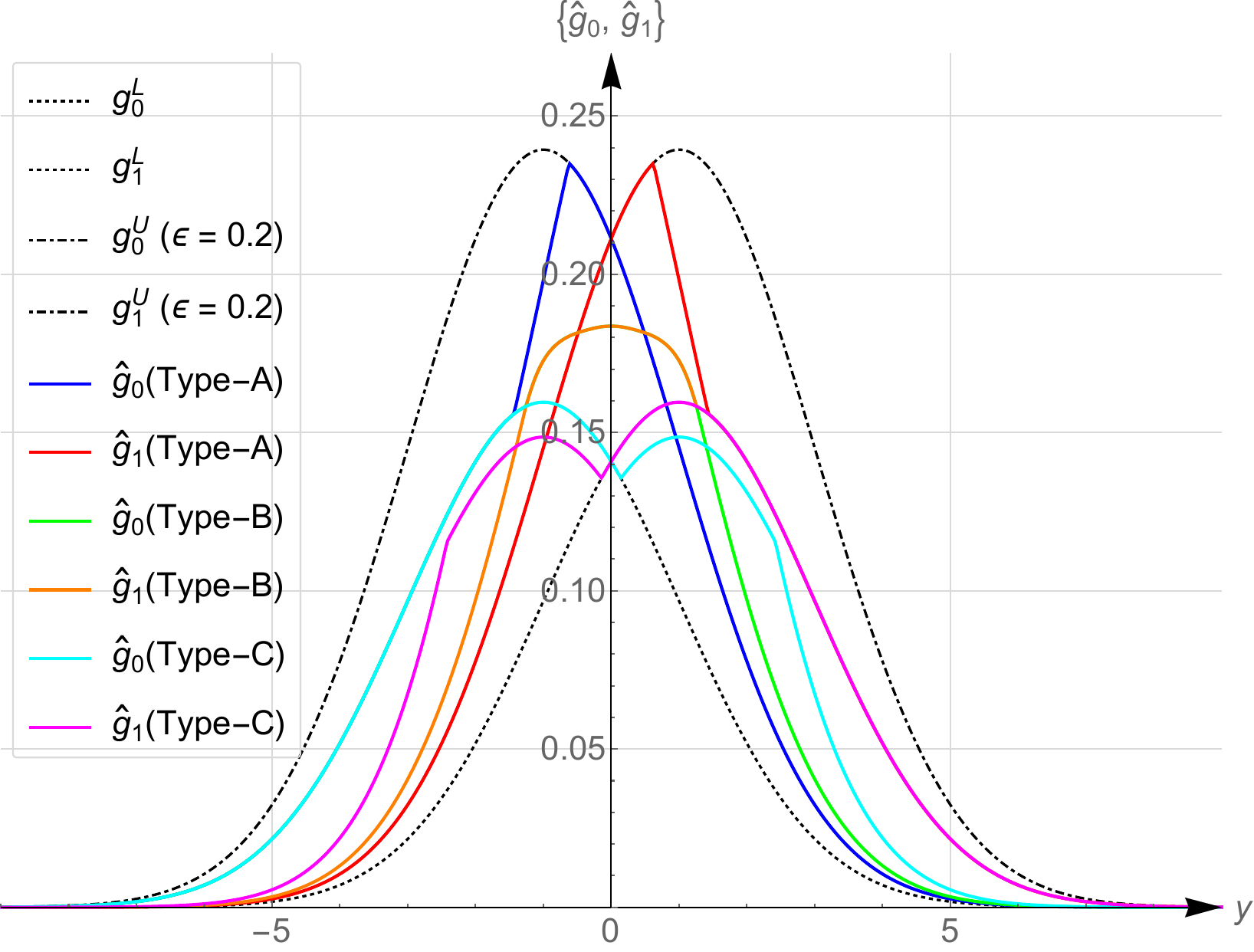}}
\caption{Three different pairs of LFDs arising from the band model together with the bounding functions for $\varepsilon\in\{0.2,0.5,1.5\}$.\label{fig9}}
\end{figure}

\begin{figure}[ttt]
  \centering
  \centerline{\includegraphics[width=8.8cm]{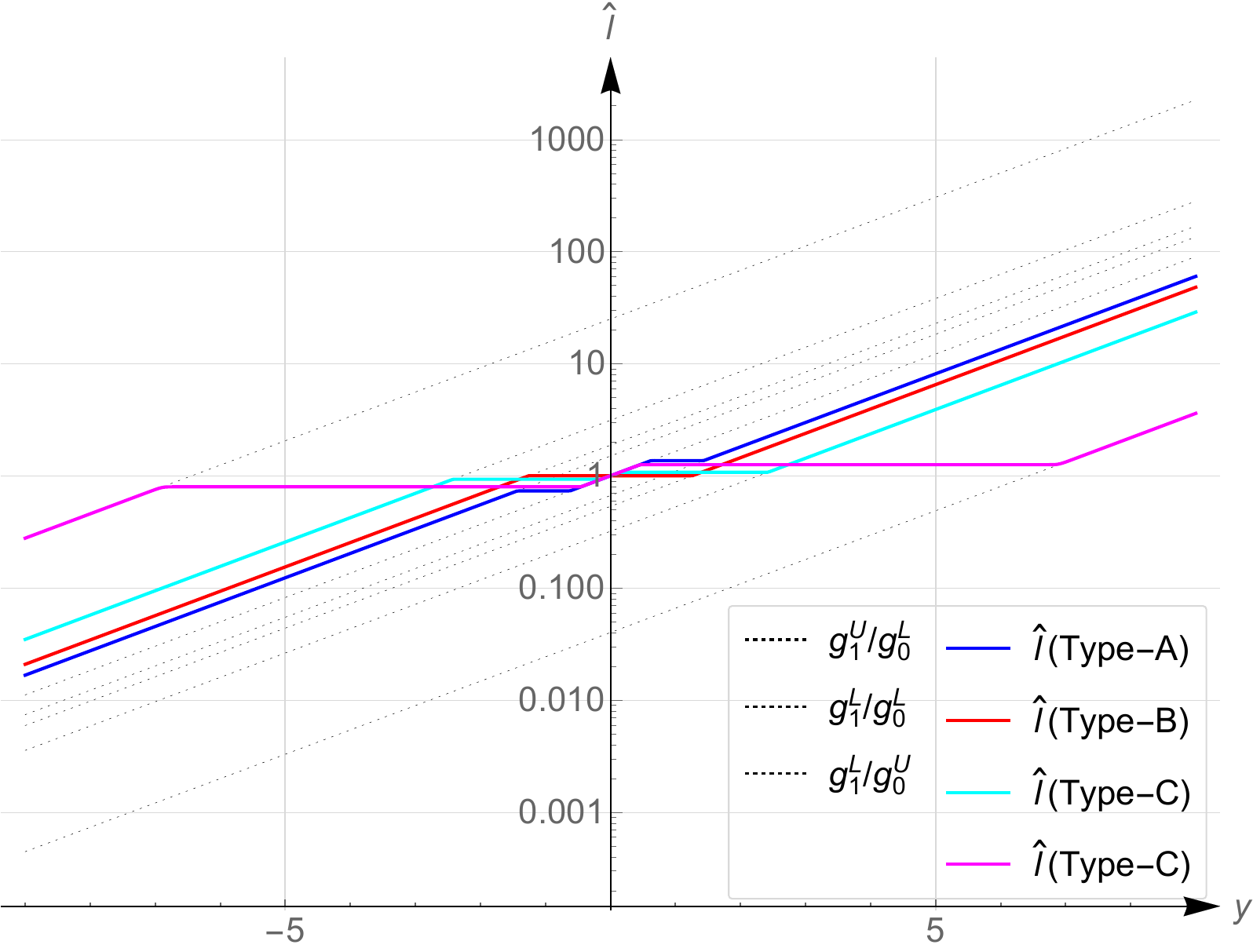}}
\caption{Three different types of Robust LRFs arising from the band model for $\varepsilon\in\{0.2,0.5,1.5,19\}$ together with the nominal LRF.\label{fig10}}
\end{figure}

\begin{figure}[ttt]
  \centering
  \centerline{\includegraphics[width=8.8cm]{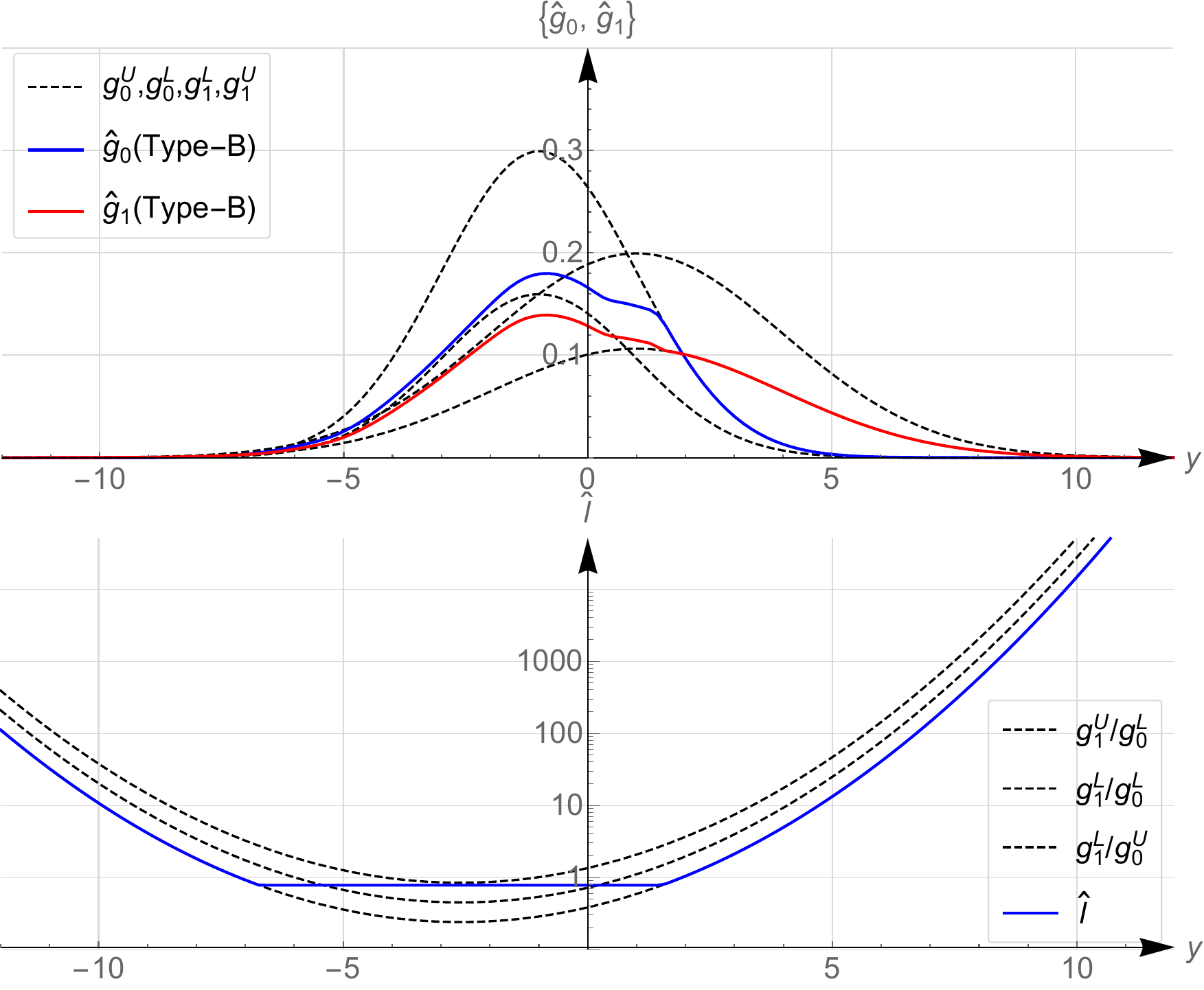}}
\caption{Degenerate Type-B least favorable densities (top) and likelihood ratio functions (bottom) for band model with asymmetric variances.\label{fig88}}
\end{figure}

\subsection{Moment Classes}
The LFDs and robust LRFs arising from the moment classes can be exemplified by solving the convex optimization problem numerically by using Algorithm~\ref{alg:lfd_design} where the class constraints
\begin{align*}
-2 \leq &\mathbb{E}_{G_0}[Y] \leq -0.5,& 0.5 \leq \mathbb{E}_{G_1}[Y] \leq 2,\nonumber \\
0 \leq &\mathbb{E}_{G_0}[Y^2] \leq 2,& 2 \leq \mathbb{E}_{G_1}[Y^2] \leq 4,
\end{align*}
are defined over the first and second moments of the probability density functions. Figure~\ref{fig19} illustrates the LFDs and the corresponding robust LRF, respectively.
\begin{figure}[ttt]
  \centering
  \centerline{\includegraphics[width=8.8cm]{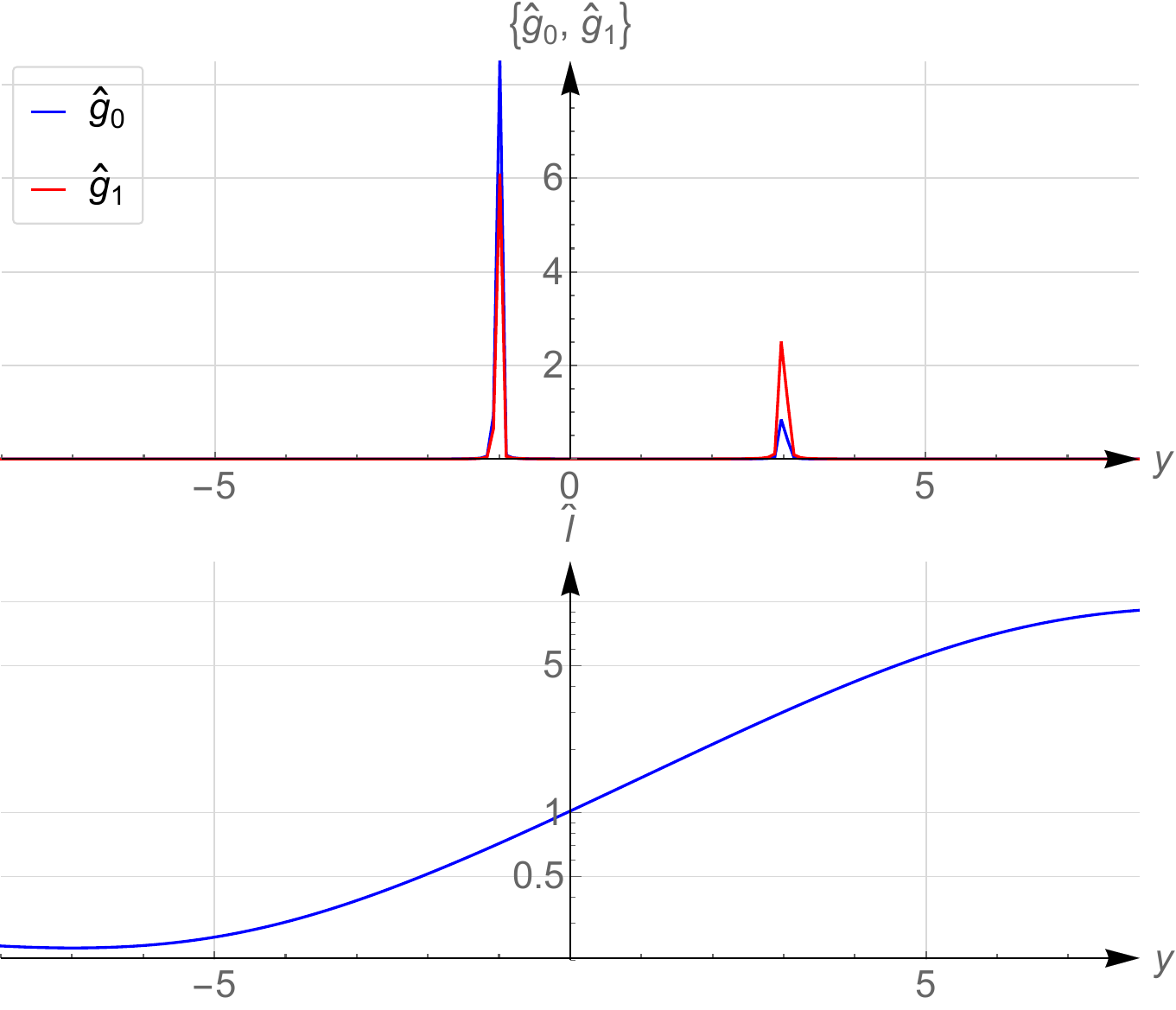}}
\caption{ Least favorable densities (top) and the likelihood ratio function (bottom) for moment-constrained uncertainty classes in the given example.\label{fig19}}
\end{figure}
\begin{figure}[ttt]
  \centering
  \centerline{\includegraphics[width=8.8cm]{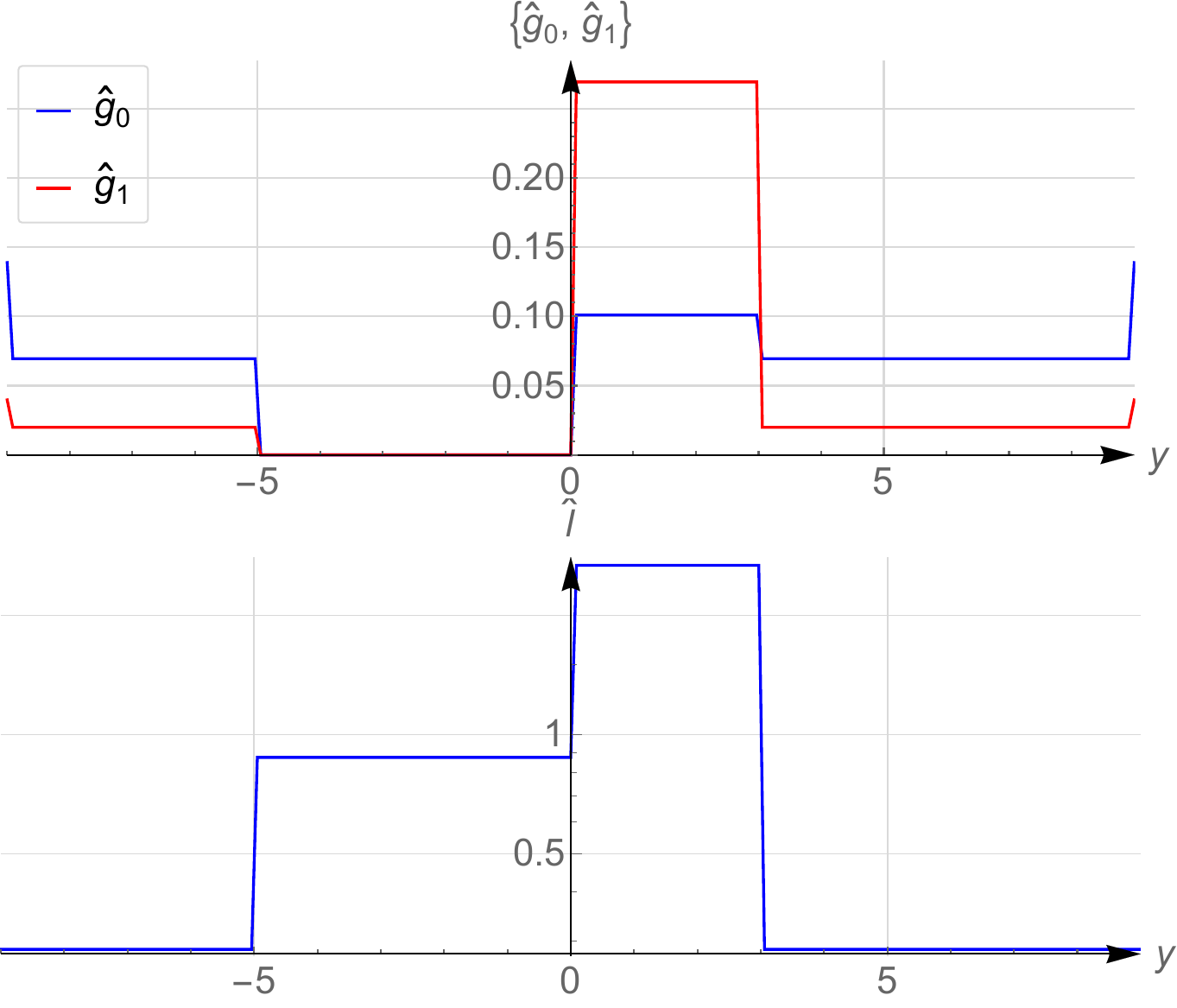}}
\vspace{-3mm}
\caption{Least favorable densities (top) and the likelihood ratio function (bottom) for p-point uncertainty classes in the given example.\label{fig21}}
\end{figure}

\subsection{P-point Classes}
Similarly, an example to the asymptotically minimax robust test arising from the p-point classes can be given. Consider the p-point classes defined by the constraints
\begin{align*}
\int_{-5}^{3} g_0(y)d y \leq 0.3,\quad \int_{0}^{3} g_1(y)d y \geq 0.8.
\end{align*}
Figure~\ref{fig21} illustrates the LFDs and the corresponding robust LRF, respectively.\\
A nice property stated by Lemma~\ref{thm:AMR_implies_FMR_conditional} can be applied here to see whether the obtained AMR tests are also FMR. For both moment classes and p-point classes the LFDs vary with varying $u$. This implies that the tests derived are only AMR and not FMR.
\section{Conclusion}\label{sec6}
This paper has established a formal equivalence between finite-sample and asymptotically minimax robust hypothesis testing. Specifically, it was shown that when a finite-sample minimax robust test exists, it coincides with the test derived via asymptotic minimax theory. This result provides a unifying perspective and enables the analytical derivation of minimax robust tests without relying on heuristic constructions.\\
As a demonstration, the total variation and band model uncertainty classes were analyzed. In both cases, least favorable distributions and the associated robust likelihood ratio functions were derived in closed parametric forms. The results generalized earlier work, notably extending Huber’s design to allow for unequal robustness parameters and offering an analytical foundation for designs previously constructed heuristically.\\
Beyond these specific models, the findings in this paper suggest that asymptotic theory offers a broadly applicable and efficient framework for designing finite-sample minimax robust tests whenever they exist. This closes a long-standing gap in the literature and provides new tools for robust decision-making under distributional uncertainty.

\appendices\label{appx}

\section{Proof of Theorem~\ref{thm1}}\label{appendix0}
The equivalence stated by \eqref{eq13} was proven in \cite[Section 6]{hube73}\footnote{In \cite{hube73} $\Omega$ is defined to be a complete separable metrizable space. Furthermore, if all $(G_0,G_1)\in\mathscr{G}_0\times\mathscr{G}_1$ are absolutely continuous with respect to a fixed measure $\mu$, $\Omega$ may need to be finite.} for a version of $D_f$;
\begin{equation*}
D_{f^{*}}(G_0,G_1)=\int_{\Omega}f^{*}\left(\frac{g_0}{g_0+g_1}\right)(g_0+g_1)d \mu,
\end{equation*}
where $f^{*}:(0,1)\to \mathbb{R}$ is a twice continuously differentiable and convex function. Let $f^{**}(y)=f^{*}(y)-f^{*}(1/2)$ which results in $f^{**}(1/2)=0$. By \cite[Equations 7-8]{vajda}, it is known that $D_{f}=D_{f^{**}}$ using the transformation $f^{**}(t)=tf((1-t)/t)$ for $t\in(0,1)$. Hence, minimizing $D_{f}$ and $D_{f^{**}}$ are equivalent over all twice continuously differentiable and convex $f$.

\section{Proof of Lemma~\ref{lem:unique_maximizer}}\label{appendix1}
\noindent\textbf{Existence:} Since $D_u$ is upper semicontinuous on the compact set $\mathscr{G}_0 \times \mathscr{G}_1$, it attains its maximum. Let $M = \max D_u > 0$ by assumption.\\
\noindent\textbf{Uniqueness:} Suppose $(G_0^*, G_1^*)$ and $(G_0', G_1')$ are two distinct maximizers achieving $D_u = M > 0$. Since $\mathscr{G}_0 \cap \mathscr{G}_1 = \emptyset$, we cannot have $G_0^* = G_1^*$ or $G_0' = G_1'$. Thus, strict concavity of $(g_0, g_1) \mapsto g_1^u g_0^{1-u}$ implies that for $t \in (0,1)$,
\begin{equation}
D_u(t G_0 + (1-t) G_0', t G_1 + (1-t) G_1') > t D_u(G_0, G_1) + (1-t) D_u(G_0', G_1') = M,
\end{equation}
a contradiction. Hence the maximizer is unique up to $\mu$-null sets.

\section{Proof of Theorem~\ref{theorem1}}\label{appendix_band1}
\mbox{\bf 1.} The sets $\bar{A}_0$ and $\bar{A}_1$ are trivially non-empty. If not, we have $\int_{\Omega}\hat{g}_0=\int_{\Omega}g_0^L d\mu<1$ and $\int_{\Omega}\hat{g}_1 d\mu=\int_{\Omega}g_1^L d\mu<1$, which are contradictions with the fact that $\hat{g}_0$ and $\hat{g}_1$ are density functions. The set $A_0$ is also non-empty and this can be shown again with contradiction. Assume that $A_0$ is empty. In this case, $A_1$ can either be empty or non-empty. Assume that $A_1$ is also empty. Then, by \eqref{eq89}, we necessarily have $\hat{g}_0=\hat{g}_1$ a.e., which is excluded by a suitable choice of $g_0^L$ and $g_1^L$. Therefore, $A_1$ is non-empty. If $A_1$ is non-empty, then we must have $\hat{g}_0=k_2g_1^L$ on $\bar{A}_0$. If not, $\hat{g}_1/\hat{g}_0$ will not be a constant function on $\bar{A}_0\cap A_1$, which is non-empty since $\Omega=\bar{A}_0$. This again yields a contradiction with $\eqref{eq89}$. Since $\hat{g}_0=k_2g_1^L$ is defined on $\Omega$, in order to satisfy \eqref{eq89}, $\hat{g}_1$ must also be  $g_1^L$ on $\bar{A}_1$. Hence, we have $\hat{g}_1=g_1^L$ a.e. which is again a contradiction with the fact that $\int_{\Omega}\hat{g}_1=1$. Therefore, $A_0$ is non-empty. A similar analysis shows that $A_1$ is also non-empty.\\
\mbox{\bf 2.} The set $\bar{A}_0\cap \bar{A}_1$ is empty. If not, from \eqref{eq90} and \eqref{eq89} we have
\begin{equation}\label{eq93}
\frac{\hat{g}_1}{\hat{g}_0}=\frac{k_1}{k_2}\frac{g_0^L}{g_1^L}=k_1=\frac{1}{k_2}.
\end{equation}
This implies ($\bar{A}_0\cap A_1)\cup (A_0\cap \bar{A}_1)=\Omega$, hence, both $\bar{A}_0\cap \bar{A}_1$ and $A_0\cap A_1$ are empty sets. Since, $A_0\cap A_1$ is non-empty, we have a contradiction, hence, $\bar{A}_0\cap \bar{A}_1$ must be empty.\\
\mbox{\bf 3.} The set $A_0\cap A_1$ is non-empty. If not, $A_0$ and $A_1$ are disjoint sets. This implies at least non-empty $\bar{A}_0\cap A_1$ and $A_0\cap \bar{A}_1$ and at most additionally non-empty $\bar{A}_0\cap \bar{A}_1$. Non-empty $\bar{A}_0\cap \bar{A}_1$ implies $\hat{g}_1/\hat{g}_0=k$ a.e on $\Omega$, see \eqref{eq93}, and this is impossible, unless $k=1$. If only $\bar{A}_0\cap A_1$ and $A_0\cap \bar{A}_1$ are non-empty, i.e. if $\bar{A}_0\cap \bar{A}_1$ and $A_0\cap A_1$ are empty, hence, $(\bar{A}_0\cap A_1) \cup (A_0\cap \bar{A}_1)=\Omega$, we have $A_0= \bar{A}_1$ and $A_1=\bar{A}_0$ together with $A_0\cup A_1=\Omega$. This is possible if and only if $k=k_1=1/k_2$, because
\begin{align*}
\bar{A}_0\cap A_1&=\{g_0>g_0^L,g_1=g_1^L\}=\{1/k_2=g_1/g_0<g_1^L/g_0^L\}, \nonumber\\
A_0\cap \bar{A}_1&=\{g_1>g_1^L,g_0=g_0^L\}=\{k_1=g_1/g_0>g_1^L/g_0^L\}.
\end{align*}
The condition $k=k_1=1/k_2$ also implies $\hat{g}_1/\hat{g}_0=1$ a.e on $\Omega$, which is avoided by suitable choices of $g_0^L$ and $g_1^L$. Hence, $A_0\cap A_1$ cannot be empty.\\
The sets $\bar{A}_0\cap A_1$ and $A_0\cap \bar{A}_1$ are both non-empty. From $A_0\cap A_1\neq\emptyset$, there are four cases  $A_0\subset A_1$, $A_1\subset A_0$, $A_0=A_1$, or $A_0\backslash A_1$ and $A_1\backslash A_0$ are both non-empty.
The first three conditions imply either non-empty $\bar{A}_0\cap \bar{A}_1$, or $A_0=\Omega$, $A_1=\Omega$ or both. The first condition is a contradiction with \eqref{eq93} and the other three imply $\hat{g}_j=g_j^L$ on $\Omega$, which is impossible, see \eqref{eq90}. Therefore, we have non-empty $A_0\cap A_1$ together with non-empty $A_0\backslash A_1$ and $A_1\backslash A_0$. This eventually implies non-empty $\bar{A}_0\cap A_1$ and $A_0\cap \bar{A}_1$.\\
It is known that $\hat{g}_1=g_1^L$ on $A_1$ and on $\bar{A}_0\cap A_1$ we have $\hat{g}_1/\hat{g}_0=1/k_2$. Hence, on $\bar{A}_0$ we must have $\hat{g}_0=k_2 g_1^L$. Similarly, on $\bar{A}_1$ we have $\hat{g}_1=k_1 g_0^L$.

\section{Proof of Corollary~\ref{corollary2}}\label{appendix_cor2}
$k_1=1/k_2$ implies empty $A_0\cap A_1$, which is impossible, and $k_1>1/k_2$ implies non-empty $(\bar{A}_0\cap A_1)\cap (A_0\cap \bar{A}_1)$, which in turn implies $k_1=1/k_2$, another contradiction. Therefore, we have $k_1<1/k_2$. Accordingly, the sets $A_0$ and $A_1$ can be written as
\begin{align*}
A_0=&(A_0\cap A_1)\cup (A_0\cap \bar{A}_1)=\{k_1\leq g_1^L/g_0^L \leq 1/k_2\}\cup \{k_1>g_1^L/g_0^L\}=\{g_1^L/g_0^L\leq 1/k_2\},\nonumber\\
A_1=&(A_0\cap A_1)\cup (\bar{A}_0\cap A_1)=\{k_1\leq g_1^L/g_0^L \leq 1/k_2\}\cup \{1/k_2<g_1^L/g_0^L\}=\{g_1^L/g_0^L\geq k_1\}.
\end{align*}

\section{Proof of Theorem~\ref{theorem2}}\label{appendix_thm2}
The definition of the sets $A_j$, their intersections, their relation to $l$, $k_1$ and $k_2$, and the fact that $k_1>1/k_2$ trivially follow from the same line of arguments used in Theorem~\ref{theorem1} and Corollary~\ref{corollary2} by considering \eqref{eq100} and \eqref{eq101}. The lower bounding function constraints are automatically satisfied as $\hat{g}_0$ and $\hat{g}_1$ are non-negative functions. The upper bounding function constraints are also satisfied in the same way as explained in Case $1$. The LFDs are obtained by unit density function constraints.

\section{Proof of Theorem~\ref{theoremeps}}\label{appendix_thmeps}
For any $g_j\in \mathscr{G}_j$, if $t>t_u$, the event $A=[\hat{l}< t]$ has a full probability and if $t\leq t_l$, it has a null probability. Therefore, \eqref{eq105} holds trivially for these cases. For $t_l<t\leq t_u$, we have
\begin{align*}
G_1(A)=&(1+\epsilon_1)F_1(A)-\epsilon_1 h \leq (1+\epsilon_1)F_1(A)=\hat{G}_1(A)\nonumber \\
G_0(A)=&(1+\epsilon_0)F_0(A)-\epsilon_0 h \geq (1+\epsilon_0)F_0(A)-\epsilon_0=1-(1+\epsilon_0)(1-F_0(A))\nonumber\\
      =&1-(1+\epsilon_0)F_0(\bar{A})=1-G_0^U(\bar{A})=1-\hat{G}_0(\bar{A})=\hat{G}_0(A).
\end{align*}
Hence, $\hat{g}_0$ and $\hat{g}_1$ are single-sample minimax robust. Moreover, by the virtue of Theorem~\ref{thm1}, single-sample minimax robust LFDs minimize all $f$-divergences, accordingly they also maximize all $u$-affinities. 

\section{Proof of Theorem~\ref{theorem3}}\label{appendix_band2}
From \eqref{eq108} and \eqref{eq109}, LFDs can be written as
\begin{equation}\label{eq116}
\hat{g}_0=\begin{cases}
g_0^L, &  A_0 \\
\frac{1}{k_2} g_1^L\,\,\mbox{or } \frac{1}{k_2} g_1^U,&  A_2\\
g_0^U, &  A_1
\end{cases},\quad
\hat{g}_1=\begin{cases}
g_1^L, &  A_3 \\
k_1 g_0^L\,\,\mbox{or } k_1 g_0^U, &  A_5\\
g_1^U, &  A_4
\end{cases},
\end{equation}
Let $\hat{g}_0=\frac{1}{k_2} g_1^L$ on $ A_2$ and $\hat{g}_1=k_1 g_1^L$ on $A_5$. Then,
\begin{equation}\label{eq117}
A_1\cap  A_5,\quad A_2\cap  A_4,\quad \mbox{and}\quad A_2\cap  A_5
\end{equation}
are all empty sets, because their existence contradicts with \eqref{eq109}. Accordingly, the robust LRF can implicitly be written as
\begin{equation*}
\quad\frac{\hat{g}_1}{\hat{g}_0}=\begin{cases}
g_1^U/g_0^L, & A_0\cap  A_4 \\
k_1, & A_0\cap  A_5 \\
g_1^L/g_0^L, & A_0\cap  A_3 \\
g_1^U/g_0^U, & A_1\cap  A_4 \\
k_2, & A_2\cap  A_3 \\
g_1^L/g_0^U, &  A_1\cap  A_3
\end{cases}.
\end{equation*}
Furthermore, from \eqref{eq108} and \eqref{eq109} we have
\begin{align}\label{eq119}
A_0\cap A_5&=\{g_1^L<g_1<g_1^U,g_0=g_0^L\}=\{g_1^L/g_0^L<k_1=g_1/g_0<g_1^U/g_0^L\}, \nonumber\\
A_2\cap A_3&=\{g_0^L<g_0<g_0^U,g_1=g_1^L\}=\{g_1^L/g_0^U<k_2=g_1/g_0<g_1^L/g_0^L\}.
\end{align}
The empty sets in \eqref{eq117} imply $A_2\subset A_3$ and $A_5\subset A_0$, which in turn imply $A_5=A_0\cap A_5$ and $A_2=A_2\cap A_3$. Accordingly, $A_2$ and $A_5$ can also be made explicit in \eqref{eq116}. The sets $A_0$, $A_1$ and $A_2$ are disjoint, as well as the sets $A_3$, $A_4$ and $A_5$. On $A_2$ we have $g_1^L/k_2<g_0^U$ and due to continuity $\frac{1}{k_2} g_1^L=g_0^U$ at least on a single point. It is also at most on a single point, if not $A_1$ and $A_2$ are not disjoint. For $A_1$, the only choice left is then $A_1=\{g_1^L/k_2\geq g_0^U\}$. Similarly, i.e. considering $g_0^L< g_1^L/k_2$ on $A_2$ etc., we have $A_0=\{g_0^L\geq g_1^L/k_2\}$. Performing the same analysis over $A_2\cap A_3$, leads to the explicit definition of the sets $A_3$, $A_4$ and $A_5$. This implies that $A_1\cap  A_4$ is an empty set. Hence, $\hat{g}_0$, $\hat{g}_1$ and $\hat{g}_1/\hat{g}_0$ follow as defined by Theorem~\ref{theorem3}, \text{Type-C}. Following the same line of arguments for the cases $\hat{g}_0=\frac{1}{k_2} g_1^U$ on $ A_2$ and $\hat{g}_1=k_1 g_1^U$ on $A_5$ we have
\begin{align*}
A_1\cap A_5&=\{g_1^L<g_1<g_1^U,g_0=g_0^U\}=\{g_1^L/g_0^U<k_1=g_1/g_0<g_1^U/g_0^U\}, \nonumber\\
A_2\cap A_4&=\{g_0^L<g_0<g_0^U,g_1=g_1^U\}=\{g_1^U/g_0^U<k_2=g_1/g_0<g_1^U/g_0^L\},
\end{align*}
in the places of $A_0\cap  A_5$ and $A_2\cap  A_3$, respectively, empty $A_0\cap A_3$, and the explicit definition of the sets $A_j$, which leads to the LRF of \text{Type-A} and the corresponding LFDs. The LRF of \text{Type-B} is a special case arising from merging the middle three regions of the LRFs of \text{Type-A} and \text{Type-C} as $k_2\to k_1$. Moreover, LRFs of \text{Type-A} and \text{Type-C} tend to clipped likelihood ratio functions for $k_1$ small enough and $k_2$ large enough, and $k_1$ large enough and $k_2$ small enough, respectively. This implies empty $A_0\cap  A_4$ and $A_1\cap  A_3$.
\bibliographystyle{IEEEtran}
\bibliography{strings4}

\begin{thebibliography}{10}
\providecommand{\url}[1]{#1}
\csname url@samestyle\endcsname
\providecommand{\newblock}{\relax}
\providecommand{\bibinfo}[2]{#2}
\providecommand{\BIBentrySTDinterwordspacing}{\spaceskip=0pt\relax}
\providecommand{\BIBentryALTinterwordstretchfactor}{4}
\providecommand{\BIBentryALTinterwordspacing}{\spaceskip=\fontdimen2\font plus
\BIBentryALTinterwordstretchfactor\fontdimen3\font minus
  \fontdimen4\font\relax}
\providecommand{\BIBforeignlanguage}[2]{{%
\expandafter\ifx\csname l@#1\endcsname\relax
\typeout{** WARNING: IEEEtran.bst: No hyphenation pattern has been}%
\typeout{** loaded for the language `#1'. Using the pattern for}%
\typeout{** the default language instead.}%
\else
\language=\csname l@#1\endcsname
\fi
#2}}
\providecommand{\BIBdecl}{\relax}
\BIBdecl

\bibitem{kay}
S.~M. Kay, \emph{Fundamentals of Statistical Signal Processing, Volume 2:
  Detection Theory}.\hskip 1em plus 0.5em minus 0.4em\relax Prentice Hall PTR,
  jan 1998.

\bibitem{levy}
B.~C. Levy, \emph{Principles of Signal Detection and Parameter Estimation},
  1st~ed.\hskip 1em plus 0.5em minus 0.4em\relax Springer Publishing Company,
  Incorporated, 2008.

\bibitem{kassam}
S.~Kassam, G.~Moustakides, and J.~Shin, ``Robust detection of known signals in
  asymmetric noise,'' \emph{IEEE Transactions on Information Theory}, vol.~28,
  no.~1, pp. 84--91, January 1982.

\bibitem{elsawy}
A.~El-Sawy and V.~VandeLinde, ``Robust detection of known signals,'' \emph{IEEE
  Transactions on Information Theory}, vol.~23, no.~6, pp. 722--727, November
  1977.

\bibitem{nonparametric}
J.~D. Gibson and J.~L. Melsa, \emph{Introduction to nonparametric detection
  with applications}, ser. Mathematics in science and engineering.\hskip 1em
  plus 0.5em minus 0.4em\relax New York, San Francisco, London: Academic Press,
  1975.

\bibitem{hube81_2}
P.~J. Huber and E.~M. Ronchetti, \emph{{Robust statistics; 2nd ed.}}, ser.
  Wiley Series in Probability and Statistics.\hskip 1em plus 0.5em minus
  0.4em\relax Hoboken, NJ: Wiley, 2009.

\bibitem{Wilcoxon}
F.~Wilcoxon, ``Individual comparisons by ranking methods,'' \emph{Biometrics
  Bulletin}, vol.~1, no.~6, pp. 80--83, 1945.

\bibitem{gulbook}
G.~G{\"{u}}l, \emph{Robust and Distributed Hypothesis Testing}, ser. Lecture
  Notes in Electrical Engineering.\hskip 1em plus 0.5em minus 0.4em\relax
  Springer, 2017, vol. 414.

\bibitem{levy09}
B.~C. Levy, ``Robust hypothesis testing with a relative entropy tolerance,''
  \emph{IEEE Transactions on Information Theory}, vol.~55, no.~1, pp. 413--421,
  2009.

\bibitem{hube65}
P.~J. Huber, ``A robust version of the probability ratio test,'' \emph{Ann.
  Math. Statist.}, vol.~36, pp. 1753--1758, 1965.

\bibitem{HuberStrassen1973}
P.~J. Huber and V.~Strassen, ``Minimax tests and the neyman–pearson lemma for
  capacities,'' \emph{Annals of Statistics}, vol.~1, no.~2, pp. 251--263, 1973.

\bibitem{dabak}
A.~G. Dabak and D.~H. Johnson, ``Geometrically based robust detection,'' in
  \emph{Proceedings of the Conference on Information Sciences and Systems},
  Johns Hopkins University, Baltimore, MD, May 1994, pp. 73--77.

\bibitem{moment}
C.~Pandit and S.~Meyn, ``Worst-case large-deviation asymptotics with
  application to queueing and information theory,'' \emph{Stochastic Processes
  and their Applications}, vol. 116, no.~5, pp. 724 -- 756, 2006.

\bibitem{hube68}
P.~J. Huber and V.~Strassen, ``Robust confidence limits,'' \emph{Z.
  Wahrcheinlichkeitstheorie verw. Gebiete}, vol.~10, pp. 269--278, 1968.

\bibitem{hube73}
------, ``Minimax tests and the \mbox{Neyman-Pearson} lemma for capacities,''
  \emph{Ann. Statistics}, vol.~1, pp. 251--263, 1973.

\bibitem{gul5}
G.~G\"{u}l, ``Minimax robust decentralized hypothesis testing for parallel
  sensor networks,'' \emph{IEEE Transactions on Information Theory}, vol.~67,
  no.~1, pp. 538--548, 2021.

\bibitem{vastola}
K.~Vastola and H.~Poor, ``On the p-point uncertainty class (corresp.),''
  \emph{IEEE Transactions on Information Theory}, vol.~30, no.~2, pp. 374--376,
  March 1984.

\bibitem{kassamband}
S.~Kassam, ``Robust hypothesis testing for bounded classes of probability
  densities (corresp.),'' \emph{IEEE Transactions on Information Theory},
  vol.~27, no.~2, pp. 242--247, March 1981.

\bibitem{fauss}
M.~\mbox{Fau\ss} and A.~M. Zoubir, ``Old bands, new tracks-revisiting the band
  model for robust hypothesis testing,'' \emph{IEEE Transactions on Signal
  Processing}, vol.~64, no.~22, pp. 5875--5886, Nov 2016.

\bibitem{martin}
R.~Martin and S.~Schwartz, ``Robust detection of a known signal in nearly
  \mbox{Gaussian} noise,'' \emph{IEEE Transactions on Information Theory},
  vol.~17, no.~1, pp. 50--56, Jan 1971.

\bibitem{kassam2}
S.~Kassam and J.~Thomas, ``Asymptotically robust detection of a known signal in
  contaminated non-\mbox{Gaussian} noise,'' \emph{IEEE Transactions on
  Information Theory}, vol.~22, no.~1, pp. 22--26, January 1976.

\bibitem{martin2}
R.~Martin and C.~McGath, ``Robust detection to stochastic signals (corresp.),''
  \emph{IEEE Transactions on Information Theory}, vol.~20, no.~4, pp. 537--541,
  Jul 1974.

\bibitem{smith}
J.~E. Smith, ``Generalized chebychev inequalities: Theory and applications in
  decision analysis,'' \emph{Oper. Res.}, vol.~43, no.~5, pp. 807--825, oct
  1995.

\bibitem{brichet}
F.~Brichet and A.~Simonian, ``Conservative \mbox{Gaussian} models applied to
  measurement-based admission control,'' in \emph{Quality of Service, 1998.
  (IWQoS 98) 1998 Sixth International Workshop on}, May 1998, pp. 68--71.

\bibitem{johnson}
M.~A. Johnson and M.~R. Taaffe, ``An investigation of phase-distribution
  moment-matching algorithms for use in queueing models,'' \emph{Queueing
  Systems}, vol.~8, no.~1, pp. 129--147, Dec 1991.

\bibitem{RahimianMehrotra2019}
H.~Rahimian and S.~Mehrotra, ``Distributionally robust optimization: A
  review,'' Optimization Online, 2019.

\bibitem{GaoXie2018}
R.~Gao, Y.~Xie, L.~Xie, and H.~Xu, ``Robust hypothesis testing using
  wasserstein uncertainty sets,'' in \emph{Advances in Neural Information
  Processing Systems (NeurIPS)}, 2018.

\bibitem{WangGaoXie2024}
J.~Wang, R.~Gao, and Y.~Xie, ``Non-convex robust hypothesis testing using
  sinkhorn uncertainty sets,'' arXiv:2403.06721, 2024.

\bibitem{SunZou2022}
Z.~Sun and S.~Zou, ``Kernel robust hypothesis testing,'' \emph{arXiv preprint
  arXiv:2205.01755}, 2022.

\bibitem{SchrabKim2024}
A.~Schrab and I.~Kim, ``Robust kernel hypothesis testing under data
  corruption,'' \emph{arXiv preprint arXiv:2402.05631}, 2024.

\bibitem{WangGaoXie2022}
J.~Wang, R.~Gao, and Y.~Xie, ``A data-driven approach to robust hypothesis
  testing,'' \emph{arXiv preprint arXiv:2201.02657}, 2022.

\bibitem{PuranikMadhowPedarsani2021}
M.~Puranik, U.~Madhow, and R.~Pedarsani, ``Generalized likelihood ratio test
  for adversarially robust hypothesis testing,'' arXiv:2105.00182, 2021.

\bibitem{gularxiv}
\BIBentryALTinterwordspacing
G.~Gül, ``Asymptotically minimax robust likelihood ratio test,'' 2026.
  [Online]. Available: \url{https://arxiv.org/abs/2602.08174}
\BIBentrySTDinterwordspacing

\bibitem{veeravalli_counterexample}
A.~Magesh, Z.~Sun, V.~V. Veeravalli, and S.~Zou, ``Robust multi-hypothesis
  testing with moment-constrained uncertainty sets,'' in \emph{2024 IEEE
  International Symposium on Information Theory (ISIT)}, 2024, pp. 849--854.

\bibitem{elsawy2}
A.~El-Sawy and V.~VandeLinde, ``Robust sequential detection of signals in
  noise,'' \emph{IEEE Transactions on Information Theory}, vol.~25, no.~3, pp.
  346--353, May 1979.

\bibitem{ralph}
D.~Ralph, ``Global convergence of damped newton's method for nonsmooth
  equations via the path search,'' \emph{Math. Oper. Res.}, vol.~19, no.~2, pp.
  352--389, 1994.

\bibitem{potra}
F.~A. Potra and S.~J. Wright, ``Interior-point methods,'' \emph{J. Comput.
  Appl. Math.}, vol. 124, no. 1-2, pp. 281--302, Dec 2000.

\bibitem{gul6}
G.~G\"{u}l and A.~M. Zoubir, ``Minimax robust hypothesis testing,'' \emph{IEEE
  Transactions on Information Theory}, vol.~63, no.~9, pp. 5572--5587, 2017.

\bibitem{vajda}
F.~Osterreicher and I.~Vajda, ``Statistical information and discrimination,''
  \emph{IEEE Transactions on Information Theory}, vol.~39, no.~3, pp.
  1036--1039, May 1993.

\end{thebibliography}
\end{document}